\RequirePackage{fix-cm}
\documentclass[smallextended]{svjour3}
\smartqed
\usepackage{amsmath,amssymb,amsfonts,mathrsfs}
\usepackage{graphicx}
\usepackage{hyperref}
\usepackage[usenames,dvipsnames]{color}
\usepackage[all]{xy}
\usepackage{bbold}
\usepackage{epstopdf}
\usepackage[margin=1.25in]{geometry}
\usepackage{float}

\numberwithin{equation}{subsection}
\numberwithin{theorem}{subsection}

\everymath{\displaystyle}

\newtheorem{thm}{Theorem}[subsection]
\newtheorem{prop}[thm]{Proposition}
\newtheorem{cor}[thm]{Corollary}
\newtheorem{lem}[thm]{Lemma}

\newtheorem{defn}[thm]{Definition}

\newcommand{\cat}[1]{\mathfrak{#1}}
\newcommand{\Set}{\mathbf{Set}}
\DeclareMathOperator{\ob}{Ob}
\DeclareMathOperator{\Del}{Del}
\DeclareMathOperator{\card}{card}
\DeclareMathOperator{\isol}{isol}

\journalname{Journal of Algebraic Combinatorics}

\begin{document}

\title{Incidence hypergraphs: The categorical inconsistency of set-systems and a characterization of quiver exponentials}
\titlerunning{Incidence hypergraphs and a characterization of quiver exponentials} 

\author{Will Grilliette \and Lucas J. Rusnak}

\institute{ Texas State University \\
 \email{\{w\_g28, Lucas.Rusnak\}@txstate.edu}} 
              
\date{Received: date / Accepted: date}

\maketitle

\begin{abstract}
This paper considers the difficulty in the set-system approach to generalizing graph theory.  These difficulties arise categorically as the category of set-system hypergraphs is shown not to be cartesian closed and lacks enough projective objects, unlike the category of directed multigraphs (i.e.\ quivers). The category of incidence hypergraphs is introduced as a ``graph-like'' remedy for the set-system issues so that hypergraphs may be studied by their locally graphic behavior via homomorphisms that allow an edge of the domain to be mapped into a subset of an edge in the codomain. Moreover, it is shown that the category of quivers embeds into the category of incidence hypergraphs via a logical functor that is the inverse image of an essential geometric morphism between the topoi. Consequently, the quiver exponential is shown to be simply represented using incidence hypergraph homomorphisms.
\end{abstract}


\setcounter{tocdepth}{2}
\tableofcontents

\section{Introduction \& Background}

We examine the combinatorial and categorical differences of three well studied categories of graph-like objects and discuss the deficiencies in the set-system approach to hypergraph theory, before a fourth category is introduced as a natural hypergraphic generalization of graph theory. The categories studied are: (1) the category of quivers $\mathfrak{Q}$ (directed graphs), (2)  the category of set-system hypergraphs $\mathfrak{H}$, (3) the category of multigraphs $\mathfrak{M}$, and (4) the category of incidence hypergraphs $\mathfrak{R}$. The categories $\mathfrak{Q}$, $\mathfrak{M}$, and $\mathfrak{H}$ are well studied in theoretical computer science \cite{graphgrammars3}, \cite{graphtransformation}, \cite{GandH} and categorical graph theory \cite{brown2008}, \cite{bumby1986}, \cite{dorfler1980}, \cite{grilliette3}, \cite{Grill1}, \cite{computationalcategory}. However, the difficulties of the set-system approach to hypergraphs are resolved by $\mathfrak{R}$, and many graph theoretic results have already been generalized to hypergraphs via ``oriented hypergraphs'' in \cite{Chen2017}, \cite{OHHar}, \cite{Reff1}, \cite{AH1}, \cite{OHMTT}, \cite{OH1}, \cite{Shi1}. The nature of quiver and graph exponentials is of particular importance regarding Hedetniemi's conjecture, where in \cite{El_Zahar_1985} it was shown that the multiplicativity of $K$ is equivalent to either $G$ or $K^G$ is $K$-colorable, for every graph $G$.
The main results provide structure theorems which illustrate; (1) that the difficulty in the set-system hypergraphic approach to generalizations of graph theory are categorical in nature; (2) the idiosyncrasies of set-systems are remedied in the category of incidence structures; (3) incidence structures are a faithful generalization of quivers via a logical functor; and (4) a characterization of the edges of quiver exponentials as morphisms under the logical inclusion into the category of incidence structures.

Moreover, the classical concepts of the incidence matrix and the bipartite representation graph are shown to be related to the adjoints arising via Kan extensions of the natural logical functor between two functor categories. Section \ref{classiccat} recalls the representation of $\mathfrak{Q}$ as both a presheaf topos as well as a comma category before extending the work from \cite{dorfler1980} which examines a multi-edge generalization of the canonical set-system hypergraph from \cite{Berge2}.  The categories $\mathfrak{H}$ and $\mathfrak{M}$ are shown nearly to be topoi, sharing numerous properties with $\cat{Q}$, but both $\cat{H}$ and $\cat{M}$ fail to be cartesian closed.  Moreover, $\mathfrak{H}$ does not have enough projective objects.

Section \ref{catR} introduces the category of incidence structures $\mathfrak{R}$, which is a presheaf topos whose comma category representation is naturally related to $\mathfrak{Q}$. The Kan extensions of the natural functors from $\mathfrak{Q}$ and $\mathfrak{R}$ have intrinsic combinatorial meaning producing: complete digraphs, bouquets of loops, disjoint generators, bipartite representations, and incidence matrices. Comparing $\cat{H}$ and $\cat{R}$, it is shown that the ``natural'' functor from $\mathfrak{H}$ to $\mathfrak{R}$ that simply inserts the implied incidence between the vertex and edge is shown to be neither continuous or cocontinuous. Moreover, the incidence-forgetful operation in the reverse direction is not even functorial. Thus, the set-system approach to hypergraph theory does not categorically extend to $\mathfrak{R}$ in any meaningful way, and a more natural generalization of ``hypergraph theory'' is to pass from $\cat{Q}$ to $\cat{R}$. Finally, a natural $\cat{Q}$ to $\cat{R}$ functor is shown to be a faithful logical functor that is part of an essential, atomic, geometric morphism. The left adjoint of the logical functor is the bipartite equivalent digraph, while the right adjoint provides a characterization of quiver exponentials as incidence morphisms.

Specifically, we develop or discuss each of the functors in the diagram in Figure \ref{figMainDiag} below. Arrows with two barbs represent a functor that has both a left and right adjoints, arrows with a left (resp. right) barb are a left (resp. right) adjoints, dotted arrows have neither, while a wavy arrow is non-functorial.
\begin{figure}[H]
\centering
\begin{equation*}
\xymatrix{
&	& 	&	& \Set \ar@/^1pc/@{-_>}[dd]|-{\overrightarrow{E}{^{\star}}} \ar@/_1pc/@{-^>}[dd]|-{\overrightarrow{E}{^{\diamond}}} \ar@/^1pc/@{-_>}[ddrrrr]|-{{E}{^{\star}}}  \ar@/_3pc/@{-_>}[ddll]|-{\check{E}{^{\star}}} \ar@/_.75pc/@{-^>}[ddll]|-{\check{E}{^{\diamond}}}	&	&	&	& 	&	&\\
&	&	&	&	&	&	&	& 	&	&\\
\Set  \ar@/^1pc/@{-_>}[rr]|-{{I}^{\star}}  \ar@/_1pc/@{-^>}[rr]|-{{I}^{\diamond}}&	&	\cat{R} \ar[ll]|-{{I}}  \ar@/^1.75pc/[uurr]|-{\check{E}} \ar@/_1.75pc/[ddrr]|-{\check{V}} \ar@{<-}[rr]|-(0.5){\Upsilon} \ar@/_1pc/@{-^>}[rr]|-(0.5){\Upsilon^{\diamond}}   \ar@/^1pc/@{-_>}[rr]|-(0.5){\Upsilon^{\star}}	&	& \cat{Q}\ar@/_1pc/@{-^>}[rr]|-(0.5){U} \ar[uu]|-{\overrightarrow{E}}\ar[dd]|-{\overrightarrow{V}} 	&	&	\cat{M}\ar@/_1pc/ @{-^>}[rr]|-{N} \ar@/_1pc/@{-_>}[ll]|-(0.5){\overrightarrow{D}}	&	&	\cat{H}\ar@/_1pc/@{.>}[rr]|-{\mathcal{I}}\ar@/_2pc/ @{-^>}[uullll]|-{E} \ar@/^2pc/ [ddllll]|-(0.5){V} \ar@/_1pc/@{-_>}[ll]|-(0.5){\Del}	&	&\cat{R} \ar@/_1pc/@{~>}[ll]|-(0.5){\mathscr{F}}	\\
&	&	&	&	&	&	&	&	 &	&\\
&	&	&	& \Set \ar@/^1pc/@{-_>}[uu]|-{\overrightarrow{V}{^{\star}}} \ar@/_1pc/@{-^>}[uu]|-{\overrightarrow{V}{^{\diamond}}} \ar@/_1pc/@{-_>}[uurrrr]|-{{V^{\star}}} \ar@/_3pc/@{-^>}[uurrrr]|-{{V^{\diamond}}} \ar@/^3pc/@{-_>}[lluu]|-{\check{V}{^{\star}}} \ar@/^.75pc/@{-^>}[lluu]|-{\check{V}{^{\diamond}}}	&	&	&	& &	&\\
}
\end{equation*}%
\caption{Functorial diagram for $\mathfrak{Q}$, $\mathfrak{M}$, $\mathfrak{H}
$, \& $\mathfrak{R}$}
\label{figMainDiag}
\end{figure}
Particular attention is paid to the asymmetry of the edge functor of $\cat{H}$, the similarities between $\cat{Q}$ and $\cat{R}$, and the unresolvable comparison between $\cat{H}$ and $\cat{R}$.

\subsection{Comma Category Framework}
This subsection provides general conditions for when the canonical projections of a comma category admit adjoint functors.  Specifically, the following construction will be used to show that the vertex functor for set-system hypergraphs will admit both a left and a right adjoint, as seen in Figure \ref{figMainDiag}.

\begin{defn}[Adjoints to $P$ and $Q$]
Let $\xymatrix{\cat{A}\ar[r]^{F} & \cat{C} & \cat{B}\ar[l]_{G}}$ be functors.

\begin{enumerate}
\item If $\mathfrak{B}$ has a terminal object $\mathbb{1}_{\mathfrak{B}}$
and $G$ is continuous, then $\mathbb{1}_{\mathfrak{C}}:=G\left(\mathbb{1}_{%
\mathfrak{B}}\right)$ is a terminal object in $\mathfrak{C}$. For $C\in\ob(%
\mathfrak{C})$, let $\mathbf{1}_{C,\mathfrak{C}}\in\mathfrak{C}\left(C,%
\mathbb{1}_{\mathfrak{C}}\right)$ be the unique morphism in $\mathfrak{C}$
from $C$ to $\mathbb{1}_{\mathfrak{C}}$. Define $P^{\star}(A):=\left(A,%
\mathbf{1}_{F(A),\mathfrak{C}},\mathbb{1}_{\mathfrak{B}}\right)$ for $A\in\ob%
(\mathfrak{A})$.

\item If $F$ has a right adjoint functor $F^{\star}$, let $\theta_{C}\in%
\mathfrak{C}\left(FF^{\star}(C),C\right)$ be the counit morphism for $C\in\ob%
(\mathfrak{C})$. Define $Q^{\star}(B):=\left(F^{\star}G(B),\theta_{G(B)},B%
\right)$ for $B\in\ob(\mathfrak{B})$.

\item If $G$ has a left adjoint functor $G^{\diamond}$, let $\eta_{C}\in%
\mathfrak{C}\left(C,GG^{\diamond}(C)\right)$ be the unit morphism for $C\in%
\ob(\mathfrak{C})$. Define $P^{\diamond}(A):=\left(A,\eta_{F(A)},G^{%
\diamond}F(A)\right)$ for $A\in\ob(\mathfrak{A})$.

\item If $\mathfrak{A}$ has an initial object $\mathbb{0}_{\mathfrak{A}}$
and $F$ is cocontinuous, then $\mathbb{0}_{\mathfrak{C}}:=F\left(\mathbb{0}_{%
\mathfrak{A}}\right)$ is an initial object in $\mathfrak{C}$. For $C\in\ob(%
\mathfrak{C})$, let $\mathbf{0}_{C,\mathfrak{C}}\in\mathfrak{C}\left(\mathbb{%
0}_{\mathfrak{C}},C\right)$ be the unique morphism in $\mathfrak{C}$ from $%
\mathbb{0}_{\mathfrak{C}}$ to $C$. Define $Q^{\diamond}(B):=\left(\mathbb{0}%
_{\mathfrak{A}},\mathbf{0}_{G(B),\mathfrak{C}},B\right)$ for $B\in\ob(%
\mathfrak{B})$.
\end{enumerate}
\end{defn}

\begin{prop}[Adjoint characterizations for $P$ and $Q$]
\label{adjoints1} Let $\xymatrix{\cat{A}\ar[r]^{F} & \cat{C} &
\cat{B}\ar[l]_{G}}$ be functors.

\begin{enumerate}
\item Assume that $\mathfrak{B}$ has a terminal object and that $G$ is
continuous. If $\xymatrix{P(A',f',B')\ar[r]^(0.7){\phi} & A}\in\mathfrak{A}$, there is a unique homomorphism $\xymatrix{(A',f',B')\ar[r]^(0.6){\hat{\phi}} & P^{\star}(A)}\in(F\downarrow
G)$ such that $P\left(\hat{\phi}\right)=\phi$.

\item Assume that $F$ has a right adjoint functor. If $\xymatrix{Q(A',f',B')\ar[r]^(0.7){\varphi} & B}\in\mathfrak{B}$, there is a
unique homomorphism $\xymatrix{(A',f',B')\ar[r]^(0.6){\hat{\varphi}} & Q^{\star}(B)}%
\in(F\downarrow G)$ such that $Q\left(\hat{\varphi}\right)=\varphi$.

\item Assume that $G$ has a left adjoint functor. If $\xymatrix{A\ar[r]^(0.35){\psi} & P(A',f',B')}\in\mathfrak{A}$, there is a
unique homomorphism $\xymatrix{P^{\diamond}(A)\ar[r]^(0.45){\hat{\psi}} & (A',f',B')}%
\in(F\downarrow G)$ such that $P\left(\hat{\psi}\right)=\psi$.

\item Assume that $\mathfrak{A}$ has an initial object and that $F$ is
cocontinuous. If $\xymatrix{B\ar[r]^(0.35){\chi} & Q(A',f',B')}\in\mathfrak{B%
}$, there is a unique homomorphism $\xymatrix{Q^{\diamond}(B)\ar[r]^(0.45){\hat{\chi}} & (A',f',B')}%
\in(F\downarrow G)$ such that $Q\left(\hat{\chi}\right)=\chi$.
\end{enumerate}
\end{prop}

\begin{proof}
\begin{enumerate}

\item Let $\hat{\phi}:=\left(\phi,\mathbf{1}_{B',\cat{B}}\right)$, where $\mathbf{1}_{B',\cat{B}}$ is the unique map from $B'$ to $\mathbb{1}_{\cat{B}}$.

\item  By the universal property of $F^{\star}$, there is a unique $\xymatrix{A'\ar[r]^(0.4){\zeta} & F^{\star}G(B)}\in\cat{A}$ such that $\theta_{G(B)}\circ F(\zeta)=G(\varphi)\circ f'$.  Let $\hat{\varphi}:=\left(\zeta,\varphi\right)$.
\end{enumerate}
The proof of part (3) (resp.\ part (4)) is dual to part (2) (resp.\ part (1)).
\qed \end{proof}

\section{Classical Categories of Graphs}\label{classiccat}

\subsection{Category of Quivers $\cat{Q}$}

\subsubsection{Functor Category Representation}
Let $\cat{E}$ be the finite category drawn below.
\[\xymatrix{
1\ar@/^/[r]^{s}\ar@/_/[r]_{t}	&	0\\
}\]

Defining $\cat{Q}:=\Set^{\cat{E}}$, let $\xymatrix{\Set & \cat{Q}\ar[l]_(0.4){\overrightarrow{V}}\ar[r]^(0.4){\overrightarrow{E}} & \Set}$ be the evaluation functors at $0$ and $1$, respectively.  An object $Q$ of $\cat{Q}$ consists of two sets, $\overrightarrow{V}(Q)$ and $\overrightarrow{E}(Q)$, and a pair of functions $\sigma_{Q},\tau_{Q}:\overrightarrow{E}(Q)\to\overrightarrow{V}(Q)$.  This object is precisely a ``directed graph,'' ``oriented graph,'' or ``quiver'' as described in \cite{bumby1986}, \cite{joyofcats}, \cite{raeburn}, \cite{schiffler}.  Constructed as a functor category into $\Set$, a category of presheaves, $\cat{Q}$ inherits a deep and rich internal structure from its parent category:  completeness and cocompleteness \cite[Corollary I.2.15.4]{borceux}, a subobject classifier \cite[Lemma A1.6.6]{elephant1}, exponential objects \cite[Proposition A1.5.5]{elephant1}, a finite set of projective generators \cite[Example I.4.5.17.b]{borceux}, injective partial morphism representers \cite[Proposition A2.4.7]{elephant1}, regularity \cite[Corollary III.5.9.2]{borceux}, to name a few.  Moreover, \cite{grilliette3} demonstrated that each object of $\cat{Q}$ admits a projective cover and an injective envelope, sharpening the covering by the projective generators and the embedding into the partial morphism representer, respectively.

Letting $\cat{1}$ be the discrete category of a single object, note that the evaluation functors $\overrightarrow{V},\overrightarrow{E}:\Set^{\cat{E}}\to\Set^{\cat{1}}$ correspond to functors from $\cat{1}$ to $\cat{E}$, mapping the one object of $\cat{1}$ to either $0$ or $1$ in $\cat{E}$.  Thus, both $\overrightarrow{V}$ and $\overrightarrow{E}$ admit adjoint functors via Kan extensions \cite[Theorem I.3.7.2]{borceux}.  Explicitly, $\overrightarrow{V}$ admits a right adjoint $\overrightarrow{V}^{\star}:\Set\to\cat{Q}$ and a left adjoint $\overrightarrow{V}^{\diamond}:\Set\to\cat{Q}$ with the following action on objects:
\begin{itemize}
\item $\overrightarrow{V}^{\star}(X)=\left(X,X\times X,\pi_{1},\pi_{2}\right)$, where $\pi_{1},\pi_{2}:X\times X\to X$ are the canonical projections;
\item $\overrightarrow{V}^{\diamond}(X)=\left(X,\emptyset,\mathbf{0}_{X},\mathbf{0}_{X}\right)$, where $\mathbf{0}_{X}:\emptyset\to X$ is the empty function.
\end{itemize}
Likewise, $\overrightarrow{E}$ admits a right adjoint $\overrightarrow{E}^{\star}:\mathbf{Set}\to\mathfrak{Q}$ and a left adjoint $\overrightarrow{E}^{\diamond}:\mathbf{Set}\to\mathfrak{Q}$ with the following actions on objects:
\begin{itemize}
\item $\overrightarrow{E}^{\star}(X)=\left(\{1\},X,\mathbf{1}_{X},\mathbf{1}_{X}\right)$, where $\mathbf{1}_{X}:X\to\{1\}$ is the constant function;
\item $\overrightarrow{E}^{\diamond}(X)=\left(\{0,1\}\times X,X,\varpi_{1},\varpi_{2}\right)$, where $\varpi_{1},\varpi_{2}:X\to\{0,1\}\times X$ are the canonical inclusions.
\end{itemize}
As can be seen in simple examples, the adjoints of $\overrightarrow{V}$ and $\overrightarrow{E}$ encode the following canonical examples: the (directed) isolated set of vertices, the complete directed multigraph, the disjoint set of directed paths of length 1, and the directed bouquet of loops at a single vertex. This is collected in Table \ref{lib1table} below.
\begin{table}[H]
\centering
\renewcommand{\arraystretch}{1}
\begin{tabular}{|l|c|l|}
\hline
Functor & $\mathrm{Dom}\rightarrow \mathrm{Codom}$ & Note \\
\hline
$\overrightarrow{E}$ & $\cat{Q} \rightarrow \mathbf{Set}$ & The set of directed edges. \\ 
$\overrightarrow{E}^{\star }$ & $\mathbf{Set}\rightarrow \cat{Q}$ & Bouquet of directed loops. \\ 
$\overrightarrow{E}^{\diamond }$ & $\mathbf{Set}\rightarrow \cat{Q}$ & Disjoint copies of $1$-paths. \\
$\overrightarrow{V}$ & $\cat{Q}\rightarrow \mathbf{Set}$ & The set of vertices. \\ 
$\overrightarrow{V}^{\star }$ & $\mathbf{Set}\rightarrow \cat{Q}$ & Complete digraph. \\ 
$\overrightarrow{V}^{\diamond }$ & $\mathbf{Set}\rightarrow \cat{Q}$ & Isolated set of vertices. \\ \hline
\end{tabular}
\caption{A library of functors for $\mathfrak{Q}$}
\label{lib1table}
\end{table}
Lastly, one can improve the generation fact.  Indeed, \cite[Theorem I.2.15.6]{borceux} states that every object in $\cat{Q}$ can be realized as a quotient of a coproduct of the following two objects:  the isolated vertex $\overrightarrow{V}^{\diamond}\left(\{1\}\right)$ and the directed path of length 1, $\overrightarrow{E}^{\diamond}\left(\{1\}\right)$.  However, any generation set for $\cat{Q}$ must have two objects, meaning this generation set is minimal.

\begin{thm}[Generation of $\cat{Q}$]\label{generation-Q}
Any set of generators for $\cat{Q}$ must have at least two non-isomorphic objects.
\end{thm}

\begin{proof}
Let $\mathscr{J}\subseteq\ob(\cat{Q})$ be a family of generators for $\cat{Q}$.  Define $f,g:\{0\}\to\{0,1\}$ by $f(0):=0$ and $g(0):=1$.  As $\overrightarrow{V}^{\diamond}(f)\neq\overrightarrow{V}^{\diamond}(g)$, there are $Q_{1}\in\mathscr{J}$ and $\xymatrix{Q_{1}\ar[r]^(0.4){\varphi_{1}} & \overrightarrow{V}^{\diamond}(\{0\})}\in\cat{Q}$ such that $\overrightarrow{V}^{\diamond}(f)\circ\varphi_{1}\neq\overrightarrow{V}^{\diamond}(g)\circ\varphi_{1}$.  Note that $\overrightarrow{E}\left(\varphi_{1}\right):\overrightarrow{E}\left(Q_{1}\right)\to\overrightarrow{E}\overrightarrow{V}^{\diamond}(\{0\})=\emptyset$.  Thus, $\overrightarrow{E}\left(Q_{1}\right)=\emptyset$.

On the other hand, as $\overrightarrow{E}^{\star}(f)\neq\overrightarrow{E}^{\star}(g)$, there are $Q_{2}\in\mathscr{J}$ and $\xymatrix{Q_{2}\ar[r]^(0.4){\varphi_{2}} & \overrightarrow{E}^{\star}(\{0\})}\in\cat{Q}$ such that $\overrightarrow{E}^{\star}(f)\circ\varphi_{2}\neq\overrightarrow{E}^{\star}(g)\circ\varphi_{2}$.  Note that $\overrightarrow{V}\overrightarrow{E}^{\star}(f)=\overrightarrow{V}\overrightarrow{E}^{\star}(g)=id_{\{1\}}$, so
\[
f\circ\overrightarrow{E}\left(\varphi_{2}\right)
=\overrightarrow{E}\left(\overrightarrow{E}^{\star}(f)\circ\varphi_{2}\right)
\neq\overrightarrow{E}\left(\overrightarrow{E}^{\star}(g)\circ\varphi_{2}\right)
=g\circ E\left(\varphi_{2}\right).
\]
Thus, there is $e\in\overrightarrow{E}\left(Q_{2}\right)$ such that $\left(f\circ\overrightarrow{E}\left(\varphi_{2}\right)\right)(e)\neq\left(g\circ\overrightarrow{E}\left(\varphi_{2}\right)\right)(e)$.  Therefore, $\card\left(\overrightarrow{E}\left(Q_{2}\right)\right)\geq 1>0=\card\left(\overrightarrow{E}\left(Q_{1}\right)\right)$, showing that $Q_{1}\not\cong_{\cat{Q}}Q_{2}$.
\qed \end{proof}

\subsubsection{Comma Category Representation}

Alternatively, $\cat{Q}$ can be constructed via a comma category.  Recall the diagonal functor for the category $\mathbf{Set}$.

\begin{defn}[{Diagonal functor, \protect\cite[p.\ 62]{maclane}}]
The \emph{diagonal functor} $\Delta:\mathbf{Set}\to\mathbf{Set}\times\mathbf{%
Set}$ is defined by $\Delta(X):=(X,X)$, $\Delta(\phi):=(\phi,\phi)$.
\end{defn}

From \cite[p.\ 87]{maclane}, $\Delta$ has a right adjoint functor $\Delta^{\star}:\mathbf{Set}\times\mathbf{Set}\to\mathbf{Set}$ determined by the categorical product, the cartesian product. Composing these two functors gives $\Delta^{\star}\Delta:\mathbf{Set}\to\mathbf{Set}$ with the action $\Delta^{\star}\Delta(X)=X\times X$, $\Delta^{\star}\Delta(\phi)(x,y)=(\phi(x),\phi(y))$.  Thus, $\Delta^{\star}\Delta$ is the $2^{\text{nd}}$-power functor from \cite[Example 3.20]{joyofcats}, and the conflictingly named ``diagonal functor'' from \cite[Definition 7.4.1]{graphgrammars3}.

As in \cite[Definition 7.4.1]{graphgrammars3}, an object $Q$ of $\left(id_{\mathbf{Set}}\downarrow \Delta^{\star}\Delta\right)$ consists of two sets, $\overrightarrow{V}(Q)$ and $\overrightarrow{E}(Q)$, and a function $\overrightarrow{\epsilon}_{Q}:\overrightarrow{E}(Q)\to\overrightarrow{V}(Q)\times\overrightarrow{V}(Q)$. This object is precisely a ``directed graph'' as described in \cite[p.\ 31]{bondy-murty}. Moreover, the notion of isomorphism in this comma category matches \cite[Exercise 1.5.3]{bondy-murty} exactly.

The proof that $\left(id_{\mathbf{Set}}\downarrow \Delta^{\star}\Delta\right)$ is isomorphic to $\Set^{\cat{E}}$ follows from the universal property of the product in $\Set$ in the diagram below.
\[\xymatrix{
&	&	\overrightarrow{E}(Q)\ar[dll]_{\sigma_{Q}}\ar[drr]^{\tau_{Q}}\ar@{..>}[d]^{\exists!\overrightarrow{\epsilon}_{Q}}\\
\overrightarrow{V}(Q)	&	&	\overrightarrow{V}(Q)\times\overrightarrow{V}(Q)\ar[ll]^{\pi_{1}}\ar[rr]_{\pi_{2}}	&	&	\overrightarrow{V}(Q)\\
}\]

\subsection{Category of Set-System Hypergraphs $\cat{H}$}

\subsubsection{Construction}

To build the category of set-system hypergraphs, recall the power-set functor for the category $\mathbf{Set}$.

\begin{defn}[{Covariant power-set functor, \protect\cite[p.\ 13]{maclane}}]
The \emph{(covariant) power-set functor} $\mathcal{P}:\mathbf{Set}\to\mathbf{Set}$ is defined in the following way:
\begin{itemize}
\item $\mathcal{P}(X)$ is the power set of $X$;
\item $\mathcal{P}(\phi)(A):=\{\phi(x):x\in A\}$, the image of $A$ under $\phi$.
\end{itemize}
\end{defn}

Let $\mathfrak{H}:=\left(id_{\mathbf{Set}}\downarrow\mathcal{P}\right)$ with domain functor $E:\mathfrak{H}\to\mathbf{Set}$ and codomain functor $V:\mathfrak{H}\to\mathbf{Set}$. An object $G$ of $\mathfrak{H}$ consists of two sets, $V(G)$ and $E(G)$, and a function $\epsilon_{G}:E(G)\to\mathcal{P}V(G)$. The category $\mathfrak{H}$ contains the category $\mathbf{H}$ of hypergraphs defined in \cite[p.\ 186]{dorfler1980} as a full subcategory, but $\mathfrak{H}$ allows for empty edges as defined in \cite[\S 1.7]{duchet1995} without any alteration to the existing objects or maps. Thus, $\mathfrak{H}$ can be considered a natural extension of $\mathbf{H}$.

Note that $\mathbf{Set}$ is cocomplete and that $id_{\mathbf{Set}}$ is its own left adjoint. Invocation of Proposition \ref{adjoints1} creates adjoint functors for $V$. Explicitly, $V$ admits a right adjoint $V^{\star}:\mathbf{Set}\to\mathfrak{H}$ and a left adjoint $V^{\diamond}:\mathbf{Set}\to\mathfrak{H}$ with the following actions on objects:
\begin{itemize}
\item $V^{\star}(X)=\left(\mathcal{P}(X),id_{\mathcal{P}(X)},X\right)$;
\item $V^{\diamond}(X)=\left(\emptyset,\mathbf{0}_{\mathcal{P}(X),\mathbf{Set}},X\right)$.
\end{itemize}
As can be seen in simple examples, the adjoints of $V$ encode the following canonical examples:  the (undirected) isolated set of vertices and the complete set-system hypergraph.

Unfortunately, as $\mathcal{P}$ is not continuous, Proposition \ref{adjoints1} does not apply to $E$. However, $E$ does admit a right adjoint in the following way.

\begin{defn}[Right adjoint to $E$]
Given a set $X$, define the set-system hypergraph $E^{\star}(X):=\left(\{0,1\}\times X,\epsilon_{E^{\star}(X)},\{1\}\right)$, where
\[
\epsilon_{E^{\star}(X)}(n,x):=\left\{\begin{array}{rl}
\emptyset, & n=0,\\ 
\\
\{1\}, & n=1.\\ 
\end{array}
\right. 
\]
Define $\zeta_{X}:EE^{\star}(X)\to X$ by $\zeta_{X}(n,x):=x$.
\end{defn}

\begin{prop}[Characterization of $E^{\star}$]\label{rightadjoint-E}
If $\xymatrix{E(G)\ar[r]^(0.6){\xi} & X}\in\mathbf{Set}$, there is a unique $%
\xymatrix{G\ar[r]^(0.4){\hat{\xi}} & E^{\star}(X)}\in\mathfrak{H}$ such that 
$\zeta_{X}\circ E\left(\hat{\xi}\right)=\xi$.
\end{prop}

\begin{proof}
Define $\alpha:E(G)\to EE^{\star}(X)$ by
\[
\alpha(e):=\left\{\begin{array}{cc}
\left(0,\xi(e)\right),	&	\epsilon_{G}(e)=\emptyset,\\
\left(1,\xi(e)\right),	&	\epsilon_{G}(e)\neq\emptyset.
\end{array}\right.
\]
Let $\hat{\xi}:=\left(\alpha,\mathbf{1}_{V(G)}\right)$, where $\mathbf{1}_{V(G)}$ is the constant map from $V(G)$ to $\{1\}$.
\qed \end{proof}

Sadly, as shown Lemma \ref{topos-fail}, $E$ is not continuous and cannot admit a left adjoint functor. These functors are collected for reference in Table \ref{lib2table}.

\begin{table}[H]
\centering
\renewcommand{\arraystretch}{1}
\begin{tabular}{|l|c|l|}
\hline
Functor & $\mathrm{Dom}\rightarrow \mathrm{Codom}$ & Note \\
\hline
$E$ & $\cat{H}\rightarrow \mathbf{Set}$ & The set of hyperedges. \\ 
$E^{\star }$ & $\mathbf{Set}\rightarrow \cat{H}$ & Bouquet of undirected 1-edges and loose 0-edges. \\
$V$ & $\cat{H}\rightarrow \mathbf{Set}$ & The set of vertices. \\ 
$V^{\star }$ & $\mathbf{Set}\rightarrow \cat{H}$ & (Simplicial) Complete hypergraph. \\ 
$V^{\diamond }$ & $\mathbf{Set}\rightarrow \cat{H}$ & Isolated set of vertices. \\ \hline
\end{tabular}
\caption{A library of functors for $\mathfrak{H}$}
\label{lib2table}
\end{table}

\subsubsection{Topos-like Properties}

As $\cat{H}=\left(id_{\Set}\downarrow\mathcal{P}\right)$ differs from $\cat{Q}\cong\left(id_{\Set}\downarrow\Delta^{\star}\Delta\right)$ only in the second coordinate, the two share a substantial amount of structure.  Firstly, a monomorphism (resp.\ epimorphism) is guaranteed to be regular and is identified as a pair of one-to-one (resp.\ onto) functions, analogous to \cite[Fact 2.15]{graphtransformation} for $\cat{Q}$.  The proof mirrors case for both $\cat{Q}$ and $\Set$, so it will be omitted.

\begin{prop}[Monomorphisms, $\cat{H}$]\label{monos-h}
For $\xymatrix{G\ar[r]^{\phi} & H}\in\mathfrak{H}$, the following are equivalent:
\begin{enumerate}
\item $\phi$ is a regular monomorphism;
\item $\phi$ is an monomorphism;
\item both $V(\phi)$ and $E(\phi)$ are one-to-one.
\end{enumerate}
\end{prop}

\begin{prop}[Epimorphisms, $\cat{H}$]\label{epis-h}
For $\xymatrix{G\ar[r]^{\phi} & H}\in\mathfrak{H}$, the following are equivalent:
\begin{enumerate}
\item $\phi$ is a regular epimorphism;
\item $\phi$ is an epimorphism;
\item both $V(\phi)$ and $E(\phi)$ are onto.
\end{enumerate}
\end{prop}

Likewise, essential monomorphisms and coessential epimorphisms are very similar to \cite[Propositions 3.3.2 \& 4.2.1]{grilliette3}.  Recall the following definitions.

\begin{defn}[{Neighborhoods \& isolation, \cite[p.\ 3 \& 5]{diestel}}]
Given a set-system hypergraph $G$ and vertices $v,w\in V(G)$, $v$ is \emph{adjacent} to $w$ if there is $e\in E(G)$ such that $\{v,w\}\subseteq\epsilon_{G}(e)$.  The \emph{neighborhood} of $v$ is the set $N_{G}(v)$ of vertices adjacent to $v$ in $G$.  On the other hand, $v$ is \emph{isolated} in $G$ if $N_{G}(v)=\emptyset$.  Let $\isol(G)$ be the set of all isolated vertices in $G$.
\end{defn}

The proofs of the characterizations are nearly identical to their quiver counterparts, and will be omitted.

\begin{prop}[Essential monomorphisms, $\cat{H}$]\label{ess-mono-h}
For $\xymatrix{G\ar[r]^{\phi} & H}\in\mathfrak{H}$, $\phi$ is essential if and only if the following conditions hold:
\begin{enumerate}
\item if $V(G)\neq\emptyset$, then $V(\phi)$ is bijective;
\item if $V(G)=\emptyset$, then $\card\left(V(H)\right)\leq 1$;
\item if $S\in\mathcal{P}V(G)$ and $\epsilon_{G}^{-1}(S)\neq\emptyset$, then $\mathcal{P}E(\phi)\left(\epsilon_{G}^{-1}(S)\right)=\epsilon_{H}^{-1}\left(\mathcal{P}V(\phi)(S)\right)$;
\item if $U\in\mathcal{P}V(H)$ and $U\neq\epsilon_{H}\left(E(\phi)(g)\right)$ for all $g\in E(G)$, then $\card\left(\epsilon_{H}^{-1}(U)\right)\leq 1$.
\end{enumerate}
\end{prop}

\begin{prop}[Coessential epimorphisms, $\cat{H}$]\label{coess-epis-h}
An epimorphism $\xymatrix{G\ar@{->>}[r]^{\phi} & H}\in\mathfrak{H}$ is coessential if and only if the following conditions hold:
\begin{enumerate}
\item $E(\phi)$ is bijective;
\item if $v\in\isol(G)$, then $V(\phi)(v)\in\isol(H)$;
\item if $w\in\isol(H)$, then there is a unique $v\in\isol(G)$ such that $V(\phi)(v)=w$.
\end{enumerate}
\end{prop}

Next, $\cat{H}$ is complete and cocomplete.  Binary products and pullbacks for $\cat{H}$ were computed in \cite[p.\ 189-190]{dorfler1980}, and arbitrary products follow in direct analogy.  As such, the proof is omitted.

\begin{defn}[Construction of the product, $\cat{H}$]
Given an index set $\Lambda$, let $G_{\lambda}\in\ob(\cat{H})$ for all $\lambda\in\Lambda$. Let $Z:=\times_{\lambda\in\Lambda}E\left(G_{\lambda}\right)$ with canonical projections $r_{\lambda}:Z\to E\left(G_{\lambda}\right)$. Define a set-system hypergraph $G$ by
\begin{itemize}
\item a product vertex set $V(G):=\times_{\lambda\in\Lambda}V\left(G_{\lambda}\right)$ with canonical projections $p_{\lambda}:V(G)\to V\left(G_{\lambda}\right)$,
\item a product edge set colored by its endpoint set below, 
\[
E(G):=\left\{\left(A,\overrightarrow{e}\right)\in\mathcal{P}V(G)\times Z:\mathcal{P}\left(p_{\lambda}\right)(A)=\left(\epsilon_{G_{\lambda}}\circ r_{\lambda}\right)\left(\overrightarrow{e}\right)\forall\lambda\in\Lambda\right\}, 
\]
\item an endpoint map $\epsilon_{G}:E(G)\to\mathcal{P}V(G)$ by $\epsilon_{G}\left(A,\overrightarrow{e}\right):=A$.
\end{itemize}
Let $q_{\lambda}:E(G)\to E\left(G_{\lambda}\right)$ by $q_{\lambda}\left(A,\overrightarrow{e}\right):=r_{\lambda}\left(\overrightarrow{e}\right)$. Routine checks show that $\pi_{\lambda}:=\left(q_{\lambda},p_{\lambda}\right)$ is a morphism in $\cat{H}$ from $G$ to $G_{\lambda}$ for all $\lambda\in\Lambda$.
\end{defn}

\begin{lem}[Universal property of the product, $\mathfrak{H}$]\label{product-H}
If $\xymatrix{H\ar[r]^{\rho_{\lambda}} & G_{\lambda}}\in\mathfrak{H}$ for all $\lambda\in\Lambda$, then there is a unique $\xymatrix{H\ar[r]^{\hat{\rho}} & G}\in\mathfrak{H}$ such that $\pi_{\lambda}\circ\hat{\rho}=\rho_{\lambda}$.
\end{lem}

These properties culminate in the following theorem.

\begin{thm}[Limit properties of $\cat{H}$]\label{limits-h}
The category $\cat{H}$ is complete, cocomplete, and regular.
\end{thm}

\begin{proof}
By \cite[Theorem 3]{computationalcategory}, $\cat{H}$ is cocomplete.  Also, $\cat{H}$ is complete by applying \cite[Theorem I.2.8.1 \& Proposition I.2.8.2]{borceux} to Lemma \ref{product-H} and the pullbacks of \cite[p.\ 190]{dorfler1980}.  As $\cat{H}$ is complete, every morphism admits a kernel pair.  As $\cat{H}$ is cocomplete, the coequalizer of a kernel pair exists.  Lastly, the pullback of a regular epimorphism is a regular epimorphism again by a diagram chase as mentioned in \cite[p.\ 190]{dorfler1980}.
\qed \end{proof}

Finally, every object of $\cat{H}$ admits a partial morphism representer, constructed by appending a new vertex and a new edge for every subset of vertices.  These new components play the role of ``false'', while the original components correspond to ``true''.

\begin{defn}[Partial morphism representer, $\cat{H}$]
For $G\in\ob(\cat{H})$, define a set-system hypergraph $\tilde{G}$ by
\begin{itemize}
\item $V\left(\tilde{G}\right):=\left(\{1\}\times V(G)\right)\cup\left\{(0,0)\right\}$,
\item $E\left(\tilde{G}\right):=\left(\{1\}\times E(G)\right)\cup\left(\{0\}\times\mathcal{P}V\left(\tilde{G}\right)\right)$,
\item $\epsilon_{\tilde{G}}(n,x):=\left\{\begin{array}{rcl}
\{1\}\times\epsilon_{G}(x),	&	n=1,\\
x,	&	n=0.\\
\end{array}\right.$
\end{itemize}
Likewise, define $\xymatrix{G\ar[r]^{\eta_{G}} & \tilde{G}}\in\cat{H}$ by $V\left(\eta_{G}\right)(v):=(1,v)$, $E\left(\eta_{G}\right)(e):=(1,e)$.  By Proposition \ref{monos-h}, $\eta_{G}$ is monic in $\cat{H}$.
\end{defn}

\begin{thm}[Universal property of $\tilde{\Box}$]\label{partialmorphism-h}
If $\xymatrix{K & \textrm{ }H\ar@{>->}[l]_{\phi}\ar[r]^{\psi} & G}\in\cat{H}$ where $\phi$ is monic, then there is a unique $\xymatrix{K\ar[r]^{\hat{\psi}} & \tilde{G}}\in\cat{H}$ such that $\xymatrix{K & \textrm{ }H\ar@{>->}[l]_{\phi}\ar[r]^{\psi} & G}$ is a pullback of $\xymatrix{K\ar[r]^{\hat{\psi}} & \tilde{G} & \textrm{ }G\ar@{>->}[l]_{\eta_{G}}}$.  Consequently, if $T$ is a terminal object in $\cat{H}$, $\tilde{T}$ equipped with $\eta_{T}$ is a subobject classifier for $\cat{H}$.
\end{thm}

\begin{proof}
Define $\xymatrix{K\ar[r]^{\hat{\psi}} & \tilde{G}}\in\cat{H}$ by
\begin{itemize}
\item $V\left(\hat{\psi}\right)(v):=\left\{\begin{array}{cc}
\left(1,V(\psi)(w)\right),	&	v=V(\phi)(w),\\
\left(0,0\right),	&	\textrm{otherwise},\
\end{array}\right.$
\item $E\left(\hat{\psi}\right)(e):=\left\{\begin{array}{cc}
\left(1,E(\psi)(f)\right),	&	e=E(\phi)(f),\\
\left(0,\left(\mathcal{P}V\left(\hat{\psi}\right)\circ\epsilon_{K}\right)(e)\right),	&	\textrm{otherwise},\
\end{array}\right.$
\end{itemize}
As $\phi$ is monic, Proposition \ref{monos-h} shows that $\hat{\psi}$ is well-defined.
\qed \end{proof}

An immediate consequence is that $\cat{H}$ has enough injective objects, but moreover, injective objects can be completely characterized.  Indeed, an injective set-system hypergraph is precisely a generalized complete hypergraph.

\begin{cor}[Characterization of injective objects, $\cat{H}$]
A set-system hypergraph $G$ is injective in $\cat{H}$ if and only if the following conditions hold:
\begin{enumerate}
\item $V(G)\neq\emptyset$,
\item $\epsilon_{G}^{-1}(S)\neq\emptyset$ for all $S\in\mathcal{P}V(G)$.
\end{enumerate}
\end{cor}

\begin{proof}
$(\Rightarrow)$ As $G$ is injective with respect to monomorphisms, there is $\xymatrix{\tilde{G}\ar[r]^{\psi} & G}\in\cat{H}$ such that $\psi\circ\eta_{G}=id_{G}$.
\[\xymatrix{
G\\
G\ar[u]^{id_{G}}\ar[r]_{\eta_{G}}	&	\tilde{G}\ar@{..>}[ul]_{\exists\psi}\\
}\]
Then, $V(\psi)(0,0)\in V(G)$ and
\[\begin{array}{rcl}
\epsilon_{G}\left(E(\psi)\left(0,\{0\}\times S\right)\right)
&	=	&	\mathcal{P}V(\psi)\left(\epsilon_{\tilde{G}}\left(0,\{1\}\times S\right)\right)
=\mathcal{P}V(\psi)\left(\{1\}\times S\right)
=\mathcal{P}V\left(\psi\circ\eta_{G}\right)(S)\\
&	=	&	\mathcal{P}V\left(id_{G}\right)(S)
=id_{\mathcal{P}V(G)}(S)
=S.
\end{array}\]

$(\Leftarrow)$ Fix $v_{0}\in V(G)$ and $e_{S}\in\epsilon_{G}^{-1}(S)$ for $S\in\mathcal{P}V(G)$.  Define $\xymatrix{\tilde{G}\ar[r]^{\psi} & G}\in\cat{H}$ by
\begin{itemize}
\item $V(\psi)(n,x):=\left\{\begin{array}{cc}
x,	&	n=1,\\
v_{0},	&	n=0,
\end{array}\right.$
\item $E(\psi)(n,y):=\left\{\begin{array}{cc}
y,	&	n=1,\\
e_{\mathcal{P}V(\psi)(S)},	&	n=0.
\end{array}\right.$
\end{itemize}
Routine calculations show that $\psi\circ\eta_{G}=id_{G}$, meaning $G$ is a retract of $\tilde{G}$.
\qed \end{proof}

However, the partial morphism representer is not generally the best injective embedding, i.e.\ the injective envelope.  Instead, the latter is achieved by only appending vertices or edges where none previously existed, which is directly analogous to \cite[Definition 3.3.3]{grilliette3}.  The proof of the characterization is nearly identical to \cite[Theorem 3.3.5]{grilliette3} and will be omitted.

\begin{defn}[Loading of a set-system hypergraph]
For $G\in\ob(\cat{H})$, define the \emph{loading} of $G$ as the set-system hypergraph $L_{\cat{H}}(G)$ by
\begin{itemize}
\item $VL_{\cat{H}}(G):=\left\{\begin{array}{rcl}
V(G), & V(G)\neq\emptyset,\\
\{0\}, & V(G)=\emptyset,\\
\end{array}\right.$,
\item $EL_{\cat{H}}(G):=\left(\{1\}\times E(G)\right)\cup\left\{(0,S):S\in\mathcal{P}VL_{\cat{H}}(G),\epsilon_{G}^{-1}(S)=\emptyset\right\}$,
\item $\epsilon_{L_{\cat{H}}(G)}(n,x):=\left\{\begin{array}{rcl}
\{1\}\times\epsilon_{G}(x),	&	n=1,\\
x,	&	n=0.\\
\end{array}\right.$
\end{itemize}
Likewise, define $\xymatrix{G\ar[r]^{j_{G}} & L_{\cat{H}}(G)}\in\cat{H}$ by $V\left(j_{G}\right)(v):=v$, $E\left(j_{G}\right)(e):=(1,e)$.  By Proposition \ref{ess-mono-h}, $j_{G}$ is an essential monomorphism in $\cat{H}$.
\end{defn}

\begin{thm}[Injective envelope, $\cat{H}$]
For a set-system hypergraph $G$, $L_{\cat{H}}(G)$ equipped with $j_{G}$ is an injective envelope of $G$ in $\cat{H}$.
\end{thm}

\subsubsection{Failure of Exponentials and Projectives}

Considering the topos-like properties, one could be forgiven for assuming that $\cat{H}$ itself was a topos.  Unfortunately, it is not, failing only requirement to be cartesian closed.

\begin{lem}[Exponential failure]\label{topos-fail}
The category $\cat{H}$ is not cartesian closed, and $E$ is not continuous.
\end{lem}

\begin{proof}
We construct a counterexample for these conditions similar to \cite[Counterexample 5.1]{brown2008}

Let $P_{1}$ be the path of length 1. There is a unique map $\alpha$ from $V^{\diamond}(\{0\})$ to $P_{1}$ mapping $0$ to $v$, and a unique map $\beta$ from $V^{\diamond}(\{0\})$ to $P_{1}$ mapping $0$ to $w$.  The coequalizer of $\alpha$ and $\beta$ in $\mathfrak{H}$ appears below, quotienting $v$ and $w$ together into the set-system hypergraph $H$.
\[
\includegraphics[scale=1]{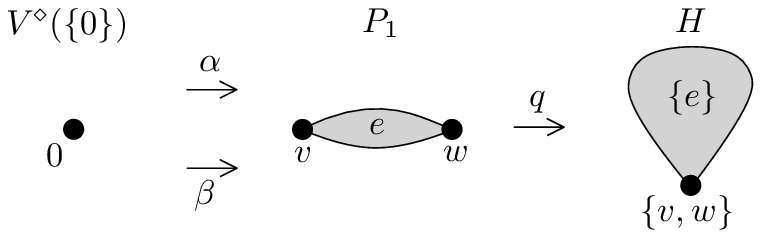}
\]
Applying the functor $P_{1}{\prod}^{\cat{H}}-$ to the diagram above, consider the coequalizer of $P_{1}{\prod}^{\cat{H}}\alpha$ and $P_{1}{\prod}^{\cat{H}}\beta$ in $\cat{H}$. Here, the vertices are quotiented, but the edges are not, giving the set-system hypergraph $K$ below.
\[
\includegraphics[scale=1]{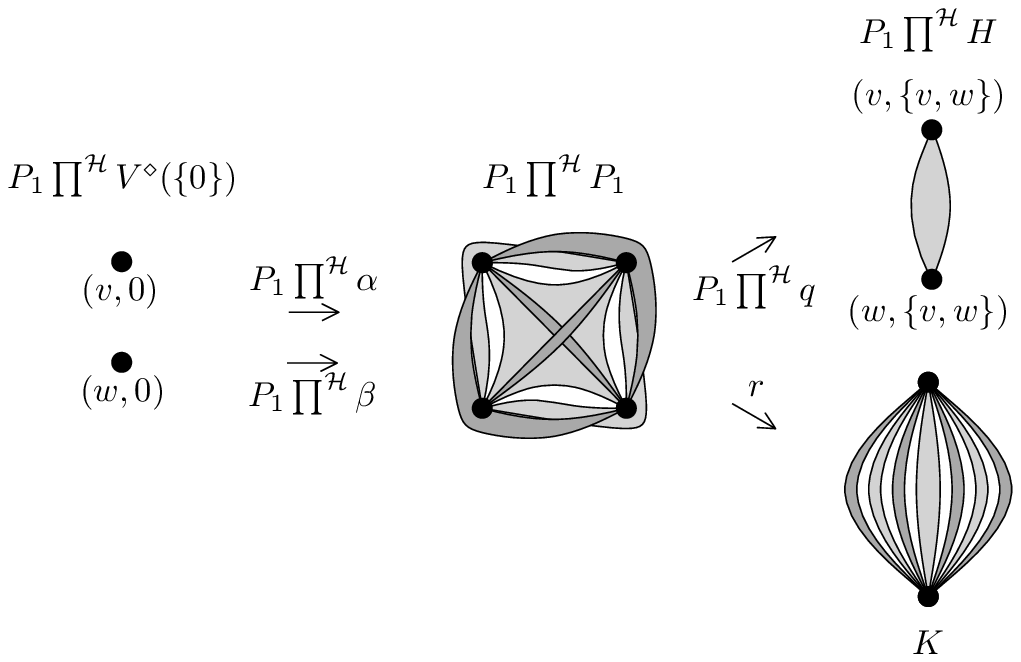}
\]
Observe that $K\not\cong_{\mathfrak{H}}P_{1}{\prod}^{\mathfrak{H}}H$ and $E\left(P_{1}\right){\prod}^{\mathbf{Set}}E\left(P_{1}\right)\not\cong_{\mathbf{Set}}E\left(P_{1}{\prod}^{\mathfrak{H}}P_{1}\right)$.
\qed \end{proof}

Also, $\cat{H}$ has another failing.  Projective objects in $\cat{H}$ are very degenerate, composed only of isolated vertices and 0-edges.

\begin{thm}[Projective objects, $\mathfrak{H}$]\label{projectives-h}
A set-system hypergraph $P$ is projective in $\mathfrak{H}$ if and only if $\epsilon_{P}(e)=\emptyset$ for all $e\in E(P)$.
\end{thm}

\begin{proof}

$(\Leftarrow)$ Consider the diagram below in $\cat{H}$, where $\phi$ is epic.
\[\xymatrix{
P\ar[d]_{\psi}\\
H	&	G\ar@{->>}[l]^{\phi}\\
}\]
By Proposition \ref{epis-h}, both $E(\phi$) and $V(\phi)$ are onto.  For each $v\in V(P)$, choose $w_{v}\in V(\phi)^{-1}\left(V(\psi)(v)\right)$, and define $\alpha:V(P)\to V(G)$ by $\alpha(v):=w_{v}$.  For each $e\in E(P)$, choose $f_{e}\in E(\phi)^{-1}\left(E(\psi)(e)\right)$, and define $\beta:E(P)\to E(G)$ by $\beta(e):=f_{e}$.  A routine calculation shows that $\hat{\psi}:=(\beta,\alpha)\in\cat{H}(P,G)$, and $\phi\circ\hat{\psi}=\psi$ by construction.

$(\Rightarrow)$ For purposes of contradiction, assume that there is $e\in E(P)$ and $v\in V(P)$ such that $v\in\epsilon_{P}(e)$.  For any set $S$, let $G$ be the set-system hypergraph constructed in the following way:
\begin{itemize}
\item $E(G):=E(P)$,
\item $V(G):=\left(\{0\}\times V(P)\right)\cup\left(\{1\}\times S\right)$,
\item $\epsilon_{G}(f):=\left\{\begin{array}{cc}\{0\}\times\epsilon_{P}(f), & f\neq e,\\ \left(\{0\}\times\epsilon_{P}(e)\right)\cup\left(\{1\}\times S\right), & f=e.\\\end{array}\right.$
\end{itemize}
Likewise, define $\alpha:V(G)\to V(P)$ by
\[
\alpha(n,w):=\left\{\begin{array}{cc}
w,	&	n=0,\\
v,	&	n=1,\\
\end{array}\right.
\]
and $\beta:=id_{E(P)}$.  A routine calculation shows that $\phi:=(\beta,\alpha)\in\cat{H}(G,P)$.  Both $\alpha$ and $\beta$ are onto, so $\phi$ is epic in $\cat{H}$ by Proposition \ref{epis-h}.  As $P$ is projective with respect to epimorphisms in $\cat{H}$, there is $\xymatrix{P\ar[r]^{\hat{\psi}} & G}\in\cat{H}$ such that $\phi\circ\hat{\psi}=id_{P}$.
\[\xymatrix{
P\ar[d]_{id_{P}}\ar@{..>}[dr]^{\hat{\psi}}\\
P	&	G\ar@{->>}[l]^{\phi}\\
}\]
Notice that
\[
id_{E(P)}
=E\left(id_{P}\right)
=E\left(\phi\circ\hat{\psi}\right)
=E\left(\phi\right)\circ E\left(\hat{\psi}\right)
=\beta\circ E\left(\hat{\psi}\right)
=E\left(\hat{\psi}\right),
\]
so
\[
\mathcal{P}V\left(\hat{\psi}\right)\left(\epsilon_{P}(e)\right)
=\epsilon_{G}\left(E\left(\hat{\psi}\right)(e)\right)
=\epsilon_{G}(e)
=\left(\{0\}\times\epsilon_{P}(e)\right)\cup\left(\{1\}\times S\right).
\]
Hence, $\card\left(\epsilon_{P}(e)\right)\geq\card\left(\epsilon_{P}(e)\right)+\card(S)\geq\card(S)$.  Since $S$ was arbitrary, $\epsilon_{P}(e)$ has larger cardinality than any set, including its own power set.  This contradicts Cantor's Theorem, so $v$ and $e$ cannot have existed.

\qed \end{proof}

Unfortunately, due to this degeneracy, projective covers in $\cat{H}$ rarely exist.  The reason for this behavior is the inability for a 0-edge to be mapped anywhere but to another 0-edge.  Thus, the only objects with a projective cover are themselves projective.

\begin{cor}[Epic images of projectives, $\cat{H}$]
Say $\xymatrix{P\ar@{->>}[r]^{\phi} & G}\in\cat{H}$ is epic, and $P$ is projective in $\cat{H}$. Then, $G$ is also projective in $\cat{H}$.  Consequently, $\cat{H}$ does not have enough projectives.
\end{cor}

\begin{proof}

By Proposition \ref{epis-h}, $E(\phi)$ is onto.  Given $e\in E(G)$, there is $f\in E(P)$ such that $E(\phi)(f)=e$.  By Theorem \ref{projectives-h}, $\epsilon_{P}(f)=\emptyset$, so
$
\epsilon_{G}(e)
=\left(\epsilon_{G}\circ E(\phi)\right)(f)
=\left(\mathcal{P}V(\phi)\circ\epsilon_{P}\right)(f)
=\mathcal{P}V(\phi)(\emptyset)
=\emptyset.
$
Hence, $G$ is projective with respect to epimorphisms in $\cat{H}$ by Theorem \ref{projectives-h}.

\qed \end{proof}

Moreover, $\cat{H}$ does not have a set of generators, but rather a proper class indexed by the parent category $\Set$.

\begin{defn}[Generators of $\cat{H}$]
For $S\in\ob(\Set)$, define a set-system hypergraph $G_{S}:=\left(\{1\},\epsilon_{G_{S}},S\right)$, where $\epsilon_{G_{S}}(1):=S$.  Let $\mathscr{G}:=\left\{G_{S}:S\in\ob(\Set)\right\}\cup\left\{V^{\diamond}(\{1\})\right\}$.
\end{defn}

\begin{thm}[Generation of $\cat{H}$]
The class $\mathscr{G}$ is a family of generators for $\cat{H}$.  If $\mathscr{J}$ is a family of generators for $\cat{H}$, then $\mathscr{J}$ is a proper class.  Consequently, $\mathscr{G}$ is a minimal family of generators.
\end{thm}

\begin{proof}
Say $\xymatrix{H\ar@/^/[r]^{\phi}\ar@/_/[r]_{\psi} & K}\in\cat{H}$ satisfy that $\phi\circ\varphi=\psi\circ\varphi$ for all $G\in\mathscr{G}$ and $\varphi\in\cat{H}(G,H)$.  If $v\in V(H)$, let $f:\{1\}\to V(H)$ by $f(1):=v$.  By Proposition \ref{adjoints1}, there is a unique $\xymatrix{V^{\diamond}(\{1\})\ar[r]^(0.6){\hat{f}} & H}\in\cat{H}$ such that $V\left(\hat{f}\right)=f$.  By assumption $\phi\circ\hat{f}=\psi\circ\hat{f}$, meaning
$
V(\phi)(v)
=V\left(\phi\circ\hat{f}\right)(1)
=V\left(\psi\circ\hat{f}\right)(1)
=V(\psi)(v).
$
Hence, $V(\phi)=V(\psi)$.  For $e\in E(H)$, let $S:=\epsilon_{H}(e)$ and define $\xymatrix{G_{S}\ar[r]^{\varphi} & H}\in\cat{H}$ by $V(\varphi)(v):=v$ and $E(\varphi)(1):=e$.  By assumption, $\phi\circ\varphi=\psi\circ\varphi$, meaning
$
E(\phi)(e)
=E\left(\phi\circ\varphi\right)(1)
=E\left(\psi\circ\varphi\right)(1)
=E(\psi)(v).
$
Thus, $E(\phi)=E(\psi)$, yielding that $\phi=\psi$.  Therefore, $\mathscr{G}$ is a family of generators for $\cat{H}$.

Let $\mathscr{J}\subseteq\ob(\cat{H})$ satisfy that $\mathscr{J}$ is a set.  Define $S:=\sqcup_{G\in\mathscr{J}}V(G)$ and a set-system hypergraph $K$ by
\begin{itemize}
\item $V(K):=\{(0,0)\}\cup\left(\{1\}\times\mathcal{P}(S)\right)$,
\item $E(K):=\{0,1,2,3\}$,
\item $\epsilon_{K}(n):=\left\{\begin{array}{cc}
\emptyset,	&	n=0,\\
\{(0,0)\},	&	n=1,\\
\{1\}\times\mathcal{P}(S),	&	n=2,3.\\
\end{array}\right.$
\end{itemize}
Let $H$ be the set-system subhypergraph of $K$ induced by edges 0, 1, and 2.  Define $\xymatrix{H\ar@/^/[r]^{\phi}\ar@/_/[r]_{\psi} & K}\in\cat{H}$ by $V(\phi)(v):=V(\psi)(v):=v$, $E(\phi)(n)=n$, and
\[
E(\psi)(n):=\left\{\begin{array}{cc}
0,	&	n=0,\\
1,	&	n=1,\\
3,	&	n=2,\\
2,	&	n=3.\\
\end{array}\right.
\]
If $G\in\mathscr{J}$ and $\xymatrix{G\ar[r]^{\varphi} & H}\in\cat{H}$, then
\[
\card\left(\epsilon_{H}\left(E(\varphi)(e)\right)\right)
=\card\left(\mathcal{P}V(\varphi)\left(\epsilon_{G}(e)\right)\right)
\leq\card\left(V(G)\right)
\leq\card(S)
<\card\left(\mathcal{P}(S)\right)
\]
by Cantor's Theorem.  Thus, $E(\varphi)(e)\in\{0,1\}$, meaning that $\phi\circ\varphi=\psi\circ\varphi$ by a routine calculation.  As $\phi\neq\psi$, $\mathscr{J}$ is not a family of generators for $\cat{H}$.
\qed \end{proof}

\subsection{Category of Multigraphs $\cat{M}$}

\subsubsection{Deletion Functor \& Limits}

This section considers the connection between a category of multigraphs to the larger category of hypergraphs.  Specifically, the following definition is used for a multigraph.

\begin{defn}[{Multigraph, \protect\cite[p.\ 185]{dorfler1980}}]
A \emph{multigraph} $G$ is a set-system hypergraph such that for all $e\in E(G)$, $1\leq\card\left(\epsilon_{G}(e)\right)\leq 2$. Let $\mathfrak{M}$ denote the full subcategory of $\mathfrak{H}$ consisting of multigraphs, and $\xymatrix{\cat{M}\ar[r]^{N} & \cat{H}}$ be the inclusion functor.
\end{defn}

This definition agrees with the definition of a ``graph'' in \cite[p.\ 2]{bondy-murty}. Moreover, the notion of isomorphism in $\mathfrak{M}$ matches \cite[p.\ 12]{bondy-murty} exactly.

There is a natural means to change any set-system hypergraph into a multigraph, by removing all non-traditional edges.  This deletion process constitutes a right adjoint functor to $N$.

\begin{defn}[Deletion functor]
Given a set-system hypergraph $H$, define a set-sytem multigraph $\Del(H):=\left(E\Del(H),\left.\epsilon_{H}\right|_{E\Del(H)},V(H)\right)$, where $E\Del(H):=\left\{e\in E(H):1\leq\card\left(\epsilon_{H}(e)\right)\leq 2\right\}$.  Let $\xymatrix{\Del(H)\ar[r]^(0.65){j_{H}} & H}\in\mathfrak{H}$ be the canonical inclusion homomorphism from $\Del(H)$ into $H$.
\end{defn}

\begin{thm}[Characterization of $\Del$]
If $G$ is a multigraph and $\xymatrix{G\ar[r]^{\phi} & H}\in\mathfrak{H}$, there is a unique $\xymatrix{G\ar[r]^(0.4){\hat{\phi}} & \Del(H)}\in\mathfrak{M}$ such that $j_{H}\circ\hat{\phi}=\phi$.
\end{thm}

\begin{proof}
Given $e\in E(G)$, then $1\leq\card\left(\epsilon_{G}(e)\right)\leq 2$ and $\left(\epsilon_{H}\circ E(\phi)\right)(e)=\left(\mathcal{P}V(\phi)\circ\epsilon_{G}\right)(e)$, so one has $1\leq\card\left(\left(\epsilon_{H}\circ E(\phi)\right)(e)\right)\leq 2$ also.  Hence, $E(\phi)(e)\in E\Del(H)$.  Define $\beta:=\left.E(\phi)\right|^{E\Del(H)}$ and $\hat{\phi}:=\left(\beta,V(\phi)\right)$.
\qed \end{proof}

One can quickly show that $\mathfrak{M}$ is replete in $\mathfrak{H}$, so $\mathfrak{M}$ is a coreflective subcategory of $\mathfrak{H}$.  As a consequence, the adjoint pair causes $\cat{M}$ to inherit several properties from its parent $\cat{H}$.  In particular, $\cat{M}$ fails to be a topos in the exact same way.

\begin{thm}[Limit properties of $\cat{M}$]
The category $\mathfrak{M}$ is complete, cocomplete, and regular.  Every object of $\cat{M}$ admits a partial morphism representer and an injective envelope.  However, $\cat{M}$ is not cartesian closed.  The inclusion functor $N$ fails to be continuous, but admits a right adjoint functor.  Also, the right adjoint functor $\Del$ is not cocontinuous.
\end{thm}

\begin{proof}
Applying the duals of \cite[Propositions I.3.5.3-4]{borceux}, $\mathfrak{M}$ is complete and cocomplete.  The same proof as Proposition \ref{monos-h} shows that a monomorphism in $\cat{M}$ corresponds precisely to a monomorphism in $\cat{H}$, so $N$ preserves monomorphisms.  By the dual of \cite[Proposition II.10.2]{hilton-stammbach}, $\Del$ preserves injective objects.  Thus, a check shows that applying $\Del$ to the loading of a multigraph forms the injective envelope in $\cat{M}$.  Using $\Del$ on the pullback diagram for the partial morphism representer for $\cat{H}$ in Theorem \ref{partialmorphism-h} yields the partial morphism representer for $\cat{M}$.  Regularity of $\cat{M}$ follows from the same proof as $\cat{H}$.

On the other hand, applying $\Del$ throughout Lemma \ref{topos-fail} shows that $\cat{M}$ is not cartesian closed and that $N$ is not continuous.

To show that $\Del$ is not cocontinuous, let $E_4$ be a set-system hypergraph with a single 4-edge and $E_1$ a set-system hypergraph with a single 1-edge. There is only one map $\xymatrix{E_4 \ar[r]^{\alpha} & E_1}\in\cat{H}$, mapping all vertices of $E_4$ to the one vertex of $E_1$.
\[
\includegraphics[scale=1]{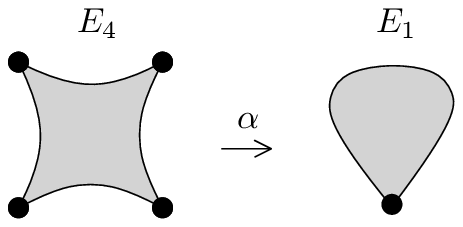}
\]
This map is epic by Proposition \ref{epis-h}, but $\Del\left(E_{4}\right)$ has no edges to map onto the single edge of $\Del\left(E_{1}\right)=E_{1}$.
\[
\includegraphics[scale=1]{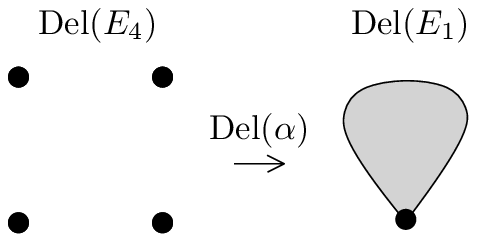}
\]
Consequently, $\Del$ does not preserve epimorphisms.
\qed \end{proof}

\subsubsection{Associated Digraph Functor}

A canonical method of reducing a quiver to a multigraph is removing the order on the endpoint map, described in \cite[p.\ 32]{bondy-murty}.  This action on objects can be extended to homomorphisms, giving a functor between the two categories.  This can also be achieved by quotienting by the symmetry relation as in \cite[p.\ 17]{brown2008}.

\begin{defn}[Underlying multigraph functor]
Given a quiver $Q$, define a set-system multigraph $U(Q):=\left(\overrightarrow{E}(Q),\epsilon_{U(Q)},\overrightarrow{V}(Q)\right)$, where $\epsilon_{U(Q)}(e):=\left\{\sigma_{Q}(e),\tau_{Q}(e)\right\}$.  Given $\xymatrix{Q\ar[r]^{\phi} & R}\in\mathfrak{Q}$, define the homomorphism $U(\phi):=\left(\overrightarrow{E}(\phi),\overrightarrow{V}(\phi)\right)=\phi$. A routine calculation shows that $U$ defines a functor from $\cat{Q}$ to $\cat{M}$.
\end{defn}

Accordingly, an \emph{orientation} of a multigraph $G$ can now be defined functorially as a quiver $Q$ such that $U(Q)=G$. On the other hand, \cite[p.\ 32]{bondy-murty} also describes a means of constructing a quiver from a multigraph by replacing an undirected edge with a pair of directed edges.  This construction creates a right adjoint functor to $U$.

\begin{defn}[Associated digraph functor]
Given a set-system multigraph $G$, define a quiver $\overrightarrow{D}(G):=\left(V(G),\overrightarrow{E}\overrightarrow{D}(G),\sigma_{\overrightarrow{D}(G)},\tau_{\overrightarrow{D}(G)}\right)$, where
\begin{itemize}
\item a 2-edge is replaced with a directed 2-cycle, and a 1-edge with a loop;
\[
\overrightarrow{E}\overrightarrow{D}(G):=\left\{(e,v,w):\epsilon_{G}(e)=\{v,w\},v\neq w\right\}
\cup\left\{(e,v,v):\epsilon_{G}(e)=\{v\}\right\}
\]
\item $\sigma_{\overrightarrow{D}(G)}(e,v,w):=v$, $\tau_{\overrightarrow{D}(G)}(e,v,w):=w$.
\end{itemize}
Likewise, define $\theta_{G}:=\left(E\left(\theta_{G}\right),id_{V(G)}\right)$, where $E\left(\theta_{G}\right)(e,v,w):=e$.  A routine calculation shows that $\theta_{G}$ is a multigraph homomorphism from $U\overrightarrow{D}(G)$ to $G$.
\end{defn}

\begin{thm}[Characterization of $\overrightarrow{D}$]\label{assoc-digraph}
If $\xymatrix{U(Q)\ar[r]^(0.6){\phi} & G}\in\mathfrak{M}$, there is a unique $\xymatrix{Q\ar[r]^(0.4){\hat{\phi}} & \overrightarrow{D}(G)}\in\mathfrak{Q}$ such that $\theta_{G}\circ U\left(\hat{\phi}\right)=\phi$.
\end{thm}

\begin{proof}

For $e\in\overrightarrow{E}(Q)$, one has
$\epsilon_{G}\left(E(\phi)(e)\right)
=\left(\mathcal{P}V(\phi)\circ\epsilon_{U(Q)}\right)(e)
=\left\{\left(V(\phi)\circ\sigma_{Q}\right)(e),\left(V(\phi)\circ\tau_{Q}\right)(e)\right\},
$
meaning $\left(E(\phi)(e),\left(V(\phi)\circ\sigma_{Q}\right)(e),\left(V(\phi)\circ\tau_{Q}\right)(e)\right)\in\overrightarrow{E}\overrightarrow{D}(G)$.  Define $\alpha:\overrightarrow{E}(Q)\to\overrightarrow{E}\overrightarrow{D}(G)$ by $\alpha(e):=\left(E(\phi)(e),\left(V(\phi)\circ\sigma_{Q}\right)(e),\left(V(\phi)\circ\tau_{Q}\right)(e)\right)$ and $\hat{\phi}:=\left(\alpha,V(\phi)\right)$.

\qed \end{proof}

However, while $\overrightarrow{D}$ has often been called the ``equivalent digraph'' operator, $\overrightarrow{D}$ and $U$ do not constitute an equivalence of categories.  This is immediately apparent as $\cat{Q}$ is a topos, and $\cat{M}$ is not.  The following lemma gives precise reasons for this failure -- $U$ does not preserve products, and $\overrightarrow{D}$ does not preserve coequalizers.

\begin{lem}\label{p1xp1}
$U$ is not continuous, and $\overrightarrow{D}$ is not cocontinuous.
\end{lem}

\begin{proof}
Consider $\overrightarrow{P}_{1}$ the directed path of length 1. By \cite[Proposition I.2.15.1]{borceux}, the product of $\overrightarrow{P}_{1}$ with itself in $\mathfrak{Q}$ is component-wise, yielding four vertices and only one edge. Consequently, $U\left(\overrightarrow{P}_{1}{\prod}^{\mathfrak{Q}}\overrightarrow{P}_{1}\right)$ has very similar structure. However, 
\[
U\left(\overrightarrow{P}_{1}\right){\prod}^{\mathfrak{M}}U\left(\overrightarrow{P}_{1}\right)
=P_{1}{\prod}^{\mathfrak{M}}P_{1}
=\Del\left(P_{1}\right){\prod}^{\mathfrak{M}}\Del\left(P_{1}\right)
=\Del\left(P_{1}{\prod}^{\mathfrak{H}}P_{1}\right), 
\]
which has four vertices and two edges from Lemma \ref{topos-fail}. 
\[
\begin{array}{cc}
U\left(\overrightarrow{P}_{1}{\prod}^{\mathfrak{Q}}\overrightarrow{P}_{1}\right) & U\left(\overrightarrow{P}_{1}\right){\prod}^{\mathfrak{M}}U\left(\overrightarrow{P}_{1}\right) \\ 
\includegraphics[scale=1]{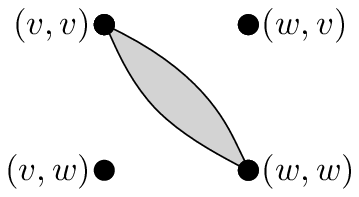} & \includegraphics[scale=1]{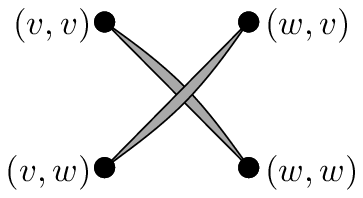} \\ 
\end{array}
\]
Hence, $U\left(\overrightarrow{P}_{1}{\prod}^{\mathfrak{Q}}\overrightarrow{P}_{1}\right)\not\cong_{%
\mathfrak{M}}U\left(\overrightarrow{P}_{1}\right){\prod}^{\mathfrak{M}}U\left(\overrightarrow{P}%
_{1}\right)$.

Let $P_{1}$, $\alpha$, $\beta$, $H$, and $q$ be as in Lemma \ref{topos-fail}. 
Applying the functor $\overrightarrow{D}$ to this diagram, the
coequalizer of $\overrightarrow{D}(\alpha)$ and $\overrightarrow{D}(\beta)$ in $\mathfrak{Q}$
again quotients vertices and not edges, giving the quiver $R$ below.
\[
\includegraphics[scale=1]{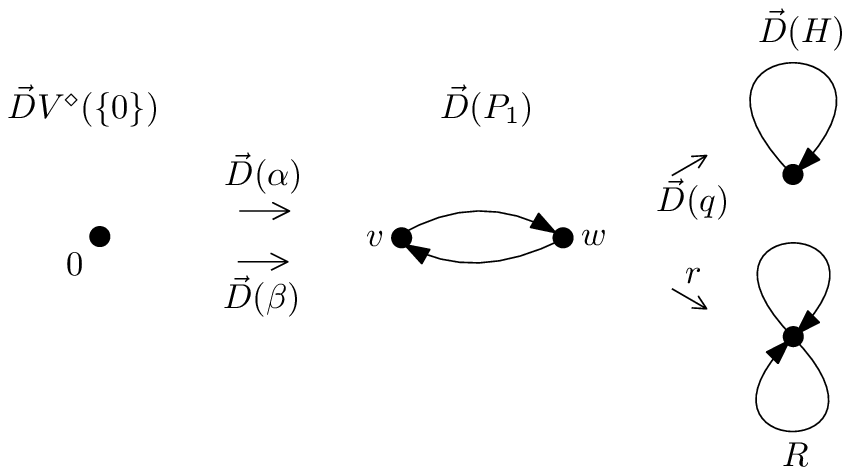}
\]
Note that $R\not\cong_{\mathfrak{Q}}\overrightarrow{D}(H)$.
\qed \end{proof}

For reference, the $\cat{H}$ and $\cat{M}$ functors from the previous two subsections appear in Table \ref{lib3table}.

\begin{table}[H]
\centering
\renewcommand{\arraystretch}{1}
\begin{tabular}{|l|c|l|}
\hline
Functor & $\mathrm{Dom}\rightarrow \mathrm{Codom}$ & Note \\
\hline\hline
$\mathrm{Del}$ & $\cat{H}\rightarrow \cat{M}$ & Deletes all edges except those of size 1 or 2.
\\ 
$N$ & $\cat{M}\rightarrow \cat{H}$ &  Natural inclusion functor of $\cat{M}$ into $\cat{H}$. \\ 
$\overrightarrow{D}$ & $\cat{M}\rightarrow \cat{Q}$ & Associated (``equivalent'') digraph. \\ 
$U$ & $\cat{Q}\rightarrow \cat{M}$ & Underlying multigraph functor. \\ \hline
\end{tabular}
\caption{A library of functors between $\mathfrak{Q}$, $\mathfrak{M}$, and $\mathfrak{H}$}
\label{lib3table}
\end{table}

\subsubsection{Projectivity}

To consider projectivity in $\cat{M}$, one must identify the epimorphisms in $\cat{M}$.  Thankfully, the cocontinuity of $VN$ and $EN$ transfer Proposition \ref{epis-h} to $\cat{M}$.  As such, the proof will be omitted.

\begin{prop}[Epimorphisms, $\cat{M}$]\label{epis-m}
For $\xymatrix{G\ar[r]^{\phi} & H}\in\cat{M}$, the following are equivalent:
\begin{enumerate}
\item $\phi$ is a regular epimorphism;
\item $\phi$ is an epimorphism;
\item both $VN(\phi)$ and $EN(\phi)$ are onto.
\end{enumerate}
\end{prop}

Projective objects cannot be obtained through $\cat{H}$.  Not only are projective objects scarce in $\cat{H}$ due to Theorem \ref{projectives-h}, but the deletion functor $\Del$ does not preserve epimorphisms as shown in Lemma \ref{p1xp1}.

On the other hand, projective covers can be constructed through $\cat{Q}$ as shown in \cite[Corollary 4.2.2]{grilliette3}.  To transfer the projective cover from $\cat{Q}$ to $\cat{M}$, observe that the associated digraph functor preserves epimorphisms.

\begin{lem}[$\overrightarrow{D}$ preserves epimorphisms]
Given an epimorphism $\xymatrix{G\ar@{->>}[r]^{\phi} & H}\in\mathfrak{M}$, $%
\overrightarrow{D}(\phi)$ is an epimorphism in $\mathfrak{Q}$.
\end{lem}

\begin{proof}

By Proposition \ref{epis-m}, both $VN(\phi)$ and $EN(\phi)$ are onto.  Given $(e,v,w)\in\overrightarrow{E}\overrightarrow{D}(H)$, recall that $e\in EN(H)$ and $v,w\in VN(H)$ satisfy $\epsilon_{H}(e)=\{v,w\}$.  There is $f\in EN(G)$ such that $EN(\phi)(f)=e$, which yields
$\mathcal{P}VN(\phi)\left(\epsilon_{G}(f)\right)
=\epsilon_{H}\left(EN(\phi)(f)\right)
=\epsilon_{H}(e)
=\{v,w\}.
$
Thus, there are $x,y\in\epsilon_{G}(f)$ such that $VN(\phi)(x)=v$ and $VN(\phi)(y)=w$.  Consequently,
$
\overrightarrow{E}\overrightarrow{D}(\phi)(f,x,y)
=\left(EN(\phi)(f),VN(\phi)(x),VN(\phi)(y)\right)
=(e,v,w),
$
showing $\overrightarrow{E}\overrightarrow{D}(\phi)$ onto.  As $\overrightarrow{V}\overrightarrow{D}(\phi)=VN(\phi)$ is onto, $\overrightarrow{D}(\phi)$ is epic in $\cat{Q}$.

\qed \end{proof}

Now, one can construct the projective cover of a multigraph in $\cat{M}$ by exploding it in $\cat{Q}$ and then removing the direction with $U$.

\begin{defn}[Explosion of a multigraph]
Given a multigraph $G$, define 
\[
X_{\mathfrak{M}}(G):=U\left(\overrightarrow{V}^{\diamond}\left(\isol(G)\right){\coprod}^{\mathfrak{Q}}\overrightarrow{E}^{\diamond}EN(G)\right). 
\]
By \cite[Proposition 4.1.1]{grilliette3} and \cite[Proposition 10.2]{hilton-stammbach}, $X_{\mathfrak{M}}(G)$ is projective in $\mathfrak{M}$.
\end{defn}

All that remains is to construct a coessential epimorphism from the explosion to cover the original multigraph.  Notably, this map is not unique due to the lack of direction in the edges of a multigraph.

\begin{thm}[Projective cover, $\mathfrak{M}$]
For a set-system multigraph $G$, there is a coessential epimorphism $\xymatrix{X_{\cat{M}}(G)\ar@{->>}[r]^(0.65){p_{G}} & G}\in\mathfrak{M}$.  Consequently, $X_{\mathfrak{M}}(G)$ equipped with $p_{G}$ is a projective cover of $G$ in $\mathfrak{M}$, and $\mathfrak{M}$ has enough projectives.
\end{thm}

\begin{proof}

To ease notation, let $A:=\overrightarrow{V}^{\diamond}\left(\isol(G)\right)$, $B:=\overrightarrow{E}^{\diamond}EN(G)$, and $C:=A{\coprod}^{\cat{Q}}B$.  Also, let $\xymatrix{A\ar[r]^{\varpi_{A}} & C & B\ar[l]_{\varpi_{B}}}\in\cat{Q}$ be the canonical inclusions.  Thus, $X_{\cat{M}}(G)=U(C)$.

Let $\iota:\isol(G)\to\overrightarrow{V}\overrightarrow{D}(G)$ be the canonical inclusion.  By Proposition \ref{adjoints1}, there is a unique $\xymatrix{A\ar[r]^(0.4){\hat{\iota}} & \overrightarrow{D}(G)}\in\cat{Q}$ such that $\overrightarrow{V}\left(\hat{\iota}\right)=\iota$.  By construction, $EN\left(\theta_{G}\right)$ is onto, so there is $f_{e}\in\overrightarrow{E}\overrightarrow{D}(G)$ such that $EN\left(\theta_{G}\right)\left(f_{e}\right)=e$ for each $e\in EN(G)$.  Define $\kappa:EN(G)\to\overrightarrow{E}\overrightarrow{D}(G)$ by $\kappa(e):=f_{e}$.  By Proposition \ref{adjoints1}, there is a unique $\xymatrix{B\ar[r]^(0.4){\hat{\kappa}} & \overrightarrow{D}(G)}\in\cat{Q}$ such that $\overrightarrow{E}\left(\hat{\kappa}\right)=\kappa$.  By the universal property of ${\coprod}^{\cat{Q}}$, there is a unique $\xymatrix{C\ar[r]^(0.4){\lambda} & \overrightarrow{D}(G)}\in\cat{Q}$ such that $\lambda\circ\varpi_{A}=\hat{\iota}$ and $\lambda\circ\varpi_{B}=\hat{\kappa}$.  Thus, define $p_{G}:=\theta_{G}\circ U(\lambda)$.

Now, $p_{G}$ is shown to be a coessential epimorphism.  Observe that
\begin{itemize}
\item $EN\left(X_{\cat{M}}(G)\right)=\{2\}\times EN(G)$,
\item $VN\left(X_{\cat{M}}(G)\right)=\left(\{1\}\times\isol(G)\right)\cup\left(\{2\}\times\{0,1\}\times EN(G)\right)$,
\item $\epsilon_{X_{\cat{M}}(G)}(2,e)=\left\{(2,0,e),(2,1,e)\right\}$ for all $e\in EN(G)$.
\end{itemize}
Unraveling the universal maps above, one has that $VN\left(p_{G}\right)(1,v)=v$ and $EN\left(p_{G}\right)(2,e)=e$ for $v\in\isol(G)$ and $e\in EN(G)$.  Consequently, $N\left(p_{G}\right)$ satisfies the conditions of Theorem \ref{coess-epis-h}, so $N\left(p_{G}\right)=p_{G}$ is a coessential epimorphism in $\cat{H}$ and, therefore, also coessential in $\cat{M}$.

\qed \end{proof}

As a projective object is isomorphic to its projective cover, the projective objects of $\cat{M}$ are completely characterized.

\begin{cor}[Projective objects, $\mathfrak{M}$]
A multigraph $P$ is projective in $\mathfrak{M}$ if and only if $P\cong_{\mathfrak{M}}U\left(\overrightarrow{V}^{\diamond}(S){\coprod}^{\mathfrak{Q}}\overrightarrow{E}^{\diamond}(T)\right)$ for some $S,T\in\ob(\mathbf{Set})$.
\end{cor}

The projectivity connection to $\cat{Q}$ yields yet another consequence.  As with $\cat{Q}$, the isolated vertex $U\overrightarrow{V}^{\diamond}(\{1\})$ and the path of length 1 $U\overrightarrow{E}^{\diamond}(\{1\})$ serve to generate the entire category $\cat{M}$.  A similar proof to Theorem \ref{generation-Q} gives the following minimality result.

\begin{thm}[Generation of $\cat{M}$]
Any set of generators for $\cat{M}$ must have at least two non-isomorphic objects.
\end{thm}

\section{Category of Incidence Hypergraphs}\label{catR}

\subsection{Functor Category Representation}
Let $\cat{D}$ be the finite category drawn below.
\[\xymatrix{
0	&	2\ar[l]_{y}\ar[r]^{z}	&	1\\
}\]
Defining $\cat{R}:=\Set^{\cat{D}}$, let $\xymatrix{\Set & \cat{R}\ar@/_/[l]_(0.4){\check{V}}\ar@/^/[l]^(0.4){\check{E}}\ar[r]^(0.4){I} & \Set}$ be the evaluation functors at $0$, $1$, and $2$, respectively.  An object $G$ of $\cat{R}$ consists of the following:  a set $I(G)$, a set $\check{V}(G)$, a set $\check{E}(G)$, a function $\varsigma_{G}:I(G)\to\check{V}(G)$, and a function $\omega_{G}:I(G)\to\check{E}(G)$.  At this point, some terminology is in order.

\begin{defn}[Graph terms]
An \emph{incidence hypergraph} is an object $G=\left(\check{V}(G),\check{E}(G),I(G),\varsigma_{G},\omega_{G}\right)$ of $\cat{R}$.
\begin{itemize}
\item The set $I(G)$ is the \emph{incidence set} of $G$, and an element $i\in I(G)$ is an \emph{incidence} of $G$.
\item The set $\check{E}(G)$ is the \emph{edge set} of $G$, and an element $e\in\check{E}(G)$ is an \emph{edge} of $G$.
\item The set $\check{V}(G)$ is the \emph{vertex set} of $G$, and an element $v\in\check{V}(G)$ is a \emph{vertex} of $G$.
\item The function $\varsigma_{G}$ is the \emph{port function} of $G$.
\item The function $\omega_{G}$ is the \emph{attachment function} of $G$.
\end{itemize}
If $\varsigma_{G}(i)=v$ and $\omega_{G}(i)=e$, then $i$ is a \emph{port} of $v$ and an \emph{attachment} of $e$, while $e$ is \emph{incident} to $v$.  A vertex $v$ is \emph{isolated} if $\varsigma_{G}^{-1}(v)=\emptyset$.  An edge $e$ is \emph{loose} if $\omega_{G}^{-1}(e)=\emptyset$.
\end{defn}

Constructed as a functor category into $\Set$, $\cat{R}$ inherits the same deep and rich internal structure from its parent category as $\cat{Q}$ did.

Again, note that the evaluation functors $\check{V},\check{E},I:\Set^{\cat{D}}\to\Set^{\cat{1}}$ correspond to functors from $\cat{1}$ to $\cat{D}$, mapping the one object of $\cat{1}$ to either $0$, $1$, or $2$ in $\cat{D}$.  Thus, $\check{V}$ admits a right adjoint $\check{V}^{\star}:\Set\to\cat{R}$ and a left adjoint $\check{V}^{\diamond}:\Set\to\cat{R}$ with the following action on objects:
\begin{itemize}
\item $\check{V}^{\star}(X)=\left(X,\{1\},X\times\{1\},\pi_{1},\pi_{2}\right)$, where $\pi_{1},\pi_{2}$ are the canonical projections;
\item $\check{V}^{\diamond}(X)=\left(X,\emptyset,\emptyset,\mathbf{0}_{X},id_{\emptyset}\right)$, where $\mathbf{0}_{X}$ is the empty function.
\end{itemize}
Likewise, $\check{E}$ admits a right adjoint $\check{E}^{\star}:\mathbf{Set}\to\cat{R}$ and a left adjoint $\check{E}^{\diamond}:\mathbf{Set}\to\cat{R}$ with the following actions on objects:
\begin{itemize}
\item $\check{E}^{\star}(X)=\left(\{1\},X,\{1\}\times X,\pi_{1},\pi_{2}\right)$, where $\pi_{1},\pi_{2}$ are the canonical projections;
\item $\check{E}^{\diamond}(X)=\left(\emptyset,X,\emptyset,id_{\emptyset},\mathbf{0}_{X}\right)$, where $\mathbf{0}_{X}$ is the empty function.
\end{itemize}
As can be seen in simple examples, the adjoints of $\check{V}$ and $\check{E}$ encode the following canonical examples:  the edge incident to $\card(X)$ vertices, the (undirected) isolated set of vertices, the bouquet of 1-edges at a single vertex, and the loose set of 0-edges.  These examples show a notable symmetry between vertices and edges in $\cat{R}$.

However, the adjoints of $I$ demonstrate behavior unlike that found in $\cat{H}$.  The right adjoint $I^{\star}:\mathbf{Set}\to\cat{R}$ and a left adjoint $I^{\diamond}:\mathbf{Set}\to\cat{R}$ have the following actions on objects:
\begin{itemize}
\item $I^{\star}(X)=\left(\{1\},\{1\},X,\mathbf{1}_{X},\mathbf{1}_{X}\right)$, where $\mathbf{1}_{X}$ is the constant function;
\item $I^{\diamond}(X)=\left(X,X,X,id_{X},id_{X}\right)$.
\end{itemize}
As can be seen in simple examples, $I^{\star}(X)$ has a single vertex and single edge with $\card(X)$ incidences attaching them, while $I^{\diamond}(X)$ consists of disjoint copies of a 1-edge.  That is, a vertex and an edge may be incident more than once.  In networking terms, this would model a single computer with multiple ports open to the same network.  The natural $\cat{R}$ functors appear in Table \ref{lib4table}.

\begin{table}[h]
\centering
\renewcommand{\arraystretch}{1}
\begin{tabular}{|l|c|l|}
\hline
Functor & $\mathrm{Dom}\rightarrow \mathrm{Codom}$ & Note \\
\hline
${I}$ & $\cat{R} \rightarrow \mathbf{Set}$ & The set of incidences. \\ 
${I}^{\star }$ & $\mathbf{Set}\rightarrow \cat{R}$ & Bouquet of incidences between $1$ vertex and $1$ edge. \\ 
${I}^{\diamond }$ & $\mathbf{Set}\rightarrow \cat{R}$ & Disjoint copies of $1$-edges (non-trivial generators). \\
$\check{E}$ & $\cat{R} \rightarrow \mathbf{Set}$ & The set of edges. \\ 
$\check{E}^{\star }$ & $\mathbf{Set}\rightarrow \cat{R}$ & Bouquet of $1$-edges at a single vertex. \\ 
$\check{E}^{\diamond }$ & $\mathbf{Set}\rightarrow \cat{R}$ & Loose $0$-edges. \\
$\check{V}$ & $\cat{R}\rightarrow \mathbf{Set}$ & The set of vertices. \\ 
$\check{V}^{\star }$ & $\mathbf{Set}\rightarrow \cat{R}$ & Bouquet of $1$-vertices on a single edge. \\ 
$\check{V}^{\diamond }$ & $\mathbf{Set}\rightarrow \cat{R}$ & Isolated set of vertices. \\ \hline
\end{tabular}
\caption{A library of functors for $\mathfrak{R}$}
\label{lib4table}
\end{table}

Lastly, \cite[Theorem I.2.15.6]{borceux} states that every object in $\cat{R}$ can be realized as a quotient of a coproduct of the following three objects:  the isolated vertex $\check{V}^{\diamond}\left(\{1\}\right)$, the loose edge $\check{E}^{\diamond}\left(\{1\}\right)$, and the 1-edge $I^{\diamond}\left(\{1\}\right)$.  Proof similar to Theorem \ref{generation-Q} shows that this generation set is minimal.

\begin{thm}[Generation of $\cat{R}$]\label{generation-R}
Any set of generators for $\cat{R}$ must have at least three non-isomorphic objects.
\end{thm}

Moreover, one can mimic \cite[Definitions 3.3.3 \& 4.1]{grilliette3} to construct an injective envelope and projective cover of an incidence hypergraph.  The constructions and proofs are nearly identical to the quiver case.

\subsection{Comma Category Representation}

Alternatively, $\cat{R}$ can be constructed via a comma category much like $\cat{Q}$.  An object $G$ of $\left(id_{\mathbf{Set}}\downarrow \Delta^{\star}\right)$ consists of the following:  a set $I(G)$, a set $\check{V}(G)$, a set $\check{E}(G)$, a function $\iota_{G}:I(G)\to\check{V}(G)\times\check{E}(G)$. This object is precisely a ``oriented hypergraph'' as described in \cite[p.\ 2]{Chen2017} without the orientation function and without restriction on isolated vertices or 0-edges. Moreover, the notion of homomorphism in this comma category matches precisely the notion of ``incidence-preserving map'' in \cite[p.\ 4]{Chen2017}.

The proof that $\left(id_{\mathbf{Set}}\downarrow \Delta^{\star}\right)$ is isomorphic to $\Set^{\cat{D}}$ follows from the universal property of the product in $\Set$ in the diagram below.
\[\xymatrix{
&	&	I(G)\ar[dll]_{\varsigma_{G}}\ar[drr]^{\omega_{G}}\ar@{..>}[d]^{\exists!\iota_{G}}\\
\check{V}(G)	&	&	\check{V}(G)\times\check{E}(G)\ar[ll]^{\pi_{1}}\ar[rr]_{\pi_{2}}	&	&	\check{E}(G)\\
}\]

\subsection{Functorial Relationship to $\cat{H}$}

Comparing the incidence hypergraphs of $\cat{R}$ to the set-system hypergraphs of $\cat{H}$, the objects appear very similar at first, with the notable exception of the set of incidences.  However, the two categories are distinct as $\cat{R}$ is cartesian closed by \cite[Proposition A1.5.5]{elephant1}, while $\cat{H}$ is not by Lemma \ref{topos-fail}.  Thus, there is no possibility of an equivalence between these two categories.

That said, one can define an obvious functor between them.  Traditionally, a vertex $v$ and an edge $e$ are ``incident'' if $v\in\epsilon_{G}(e)$ as in \cite[p.\ 3]{bondy-murty}.  This classical definition gives an incidence structure to any set-system hypergraph that can be extended to homomorphisms.

\begin{defn}[Incidence-forming functor]
Given a set-system hypergraph $G$, define an incidence hypergraph $\mathcal{I}(G):=\left(V(G),E(G),I\mathcal{I}(G),\varsigma_{\mathcal{I}(G)},\omega_{\mathcal{I}(G)}\right)$, where
\begin{itemize}
\item $I\mathcal{I}(G):=\left\{(v,e)\in V(G)\times E(G):v\in\epsilon_{G}(e)\right\}$,
\item $\varsigma_{\mathcal{I}(G)}(v,e):=v$, $\omega_{\mathcal{I}(G)}(v,e):=e$.
\end{itemize}
Given $\xymatrix{G\ar[r]^{\phi} & H}\in\mathfrak{H}$, define $I\mathcal{I}(\phi):I\mathcal{I}(G)\to I\mathcal{I}(H)$ by 
$
I\mathcal{I}(\phi)(v,e):=\left(V(\phi)(v),E(\phi)(e)\right) 
$
and $\mathcal{I}(\phi):=\left(V(\phi),E(\phi),I\mathcal{I}(\phi)\right)$.
\end{defn}

A routine calculation shows that $\mathcal{I}$ defines a functor from $\cat{H}$ to $\cat{R}$ that introduces a single incidence $(v,e)$ into a set-system hypergraph $G$ for each $v\in\epsilon_{G}(e)$ and $e\in E(G)$. Unfortunately, while $\mathcal{I}$ is a functor, it is quite poorly behaved.  Indeed, it preserves neither limits nor colimits, meaning it cannot admit any adjoint functor.

\begin{lem}[Failure of $\mathcal{I}$]\label{Ibad}
The functor $\mathcal{I}$ is neither continuous nor cocontinuous.
\end{lem}

\begin{proof}
Let $P_{1}$ be the path of length 1.  By \cite[Proposition I.2.15.1]{borceux}, the product of $\mathcal{I}\left(P_{1}\right)$ with itself in $\mathfrak{R}$ is component-wise, yielding four vertices and only one edge.  On the other hand, Lemma \ref{topos-fail} computed the product of $P_{1}$ with itself in $\cat{H}$, so consider the action of $\mathcal{I}$ upon the result. 
\[
\begin{array}{cc}
\mathcal{I}\left(P_{1}{\prod}^{\mathfrak{H}}P_{1}\right) & \mathcal{I}\left(P_{1}\right){\prod}^{\mathfrak{R}}\mathcal{I}\left(P_{1}\right) \\
\includegraphics[scale=1]{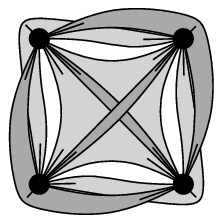} & \includegraphics[scale=1]{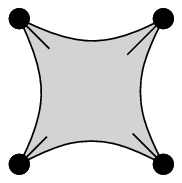} \\ 
& 
\end{array}
\]
Observe that $\mathcal{I}\left(P_{1}{\prod}^{\mathfrak{H}}P_{1}\right)\not\cong_{\mathfrak{R}}\mathcal{I}\left(P_{1}\right){\prod}^{\mathfrak{R}}\mathcal{I}\left(P_{1}\right)$.

On the other hand, let $\alpha$, $\beta$, $H$, and $q$ be as in Lemma \ref{topos-fail}. Applying the functor $\mathcal{I}$, consider the coequalizer of $\mathcal{I}(\alpha)$ and $\mathcal{I}(\beta)$ in $\cat{R}$. Here, the vertices are quotiented, but the two incidences are not, giving the incidence hypergraph $J$ below.
\[
\includegraphics[scale=1]{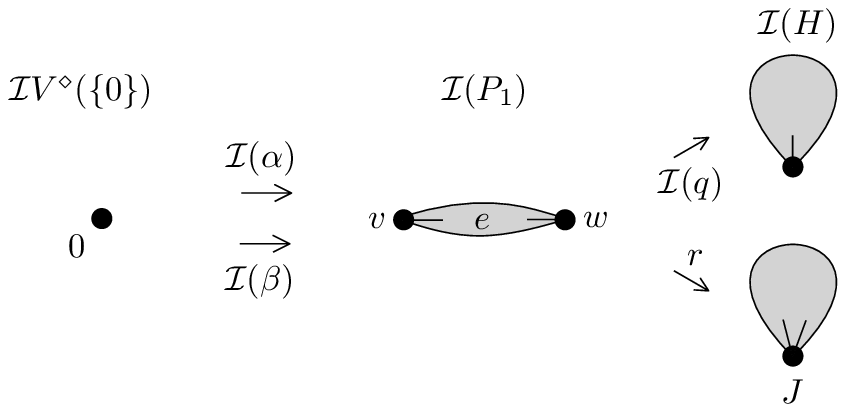}
\]
Note that $J\not\cong_{\mathfrak{R}}\mathcal{I}(H)$.
\qed \end{proof}

Likewise, there is an obvious way to change an incidence hypergraph into a set-system hypergraph, by simply defining the endpoint set of an edge to be all vertices which share an incidence with it.

\begin{defn}[Incidence-forgetting construction]
Given an incidence hypergraph $G$, define a set-system hypergraph $\mathscr{F}(G):=\left(\check{E}(G),\epsilon_{\mathscr{F}(G)},\check{V}(G)\right)$, where
$
\epsilon_{\mathscr{F}(G)}(e):=\left\{v\in\check{V}(G):\varsigma_{G}^{-1}(v)\cap\omega_{G}^{-1}(e)\neq\emptyset\right\}.
$
\end{defn}

One can quickly show that $\mathscr{F}\left(\mathcal{I}(G)\right)=G$ for every $G\in\ob(\mathfrak{H})$, but unfortunately, this map of objects cannot be made into a functor.  The reason is that in $\cat{R}$, an edge can be mapped to another edge without covering all of the latter's endpoints.  That is, homomorphisms in $\cat{R}$ allow ``locally graphic'' behavior as described in \cite{Chen2017}, \cite{OHHar}, \cite{AH1}, \cite{OH1}.  However, the structure of $\cat{H}$ cannot allow such behavior.

\begin{lem}[Failure of $\mathscr{F}$]\label{Fworse}
The map $\mathscr{F}$ cannot be made into a functor.
\end{lem}

\begin{proof}
Let $\xymatrix{G\ar[r]^{\phi} & H}\in\cat{R}$ be the map below, where $x\mapsto v$, $f\mapsto e$, and $k\mapsto i$.
\begin{center}
\includegraphics[scale=1]{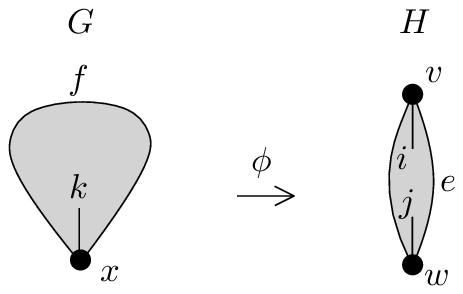}
\end{center}
If there was a map $\xymatrix{\mathscr{F}(G)\ar[r]^{\varphi} & \mathscr{F}(H)}\in\cat{H}$, then the edge $f$ can only map to the edge $e$.  Thus, one has
\[
\{v,w\}
=\epsilon_{\mathscr{F}(H)}(e)
=\left(\epsilon_{\mathscr{F}(H)}\circ E(\varphi)\right)(f)
=\left(\mathcal{P}V(\varphi)\circ\epsilon_{\mathscr{F}(G)}\right)(f)
=\mathcal{P}V(\varphi)\left(\{x\}\right)
=\left\{V(\varphi)(x)\right\},
\]
which is absurd.  Thus, there is no map from $\mathscr{F}(G)$ to $\mathscr{F}(H)$ to be assigned to $\phi$.
\qed \end{proof}

Consequently, results such as \cite[Theorem 3.1.2]{Chen2017} are likely not achievable when using the set-system hypergraphs of $\cat{H}$. The $\cat{H}$-$\cat{R}$ functors are summarized in Table \ref{lib5table} below.

\begin{table}[H]
\centering
\renewcommand{\arraystretch}{1}
\begin{tabular}{|l|c|l|}
\hline
Functor & $\mathrm{Dom}\rightarrow \mathrm{Codom}$ & Note \\
\hline\hline
$\mathcal{I}$ & $\cat{H}\rightarrow \cat{R}$ & Incidence insertion functor. \\ 
$\mathscr{F}$ & $\cat{R}\rightarrow \cat{H}$ & Incidence forgetful operation (Not
functorial). \\ \hline
\end{tabular}
\caption{A library of functors between $\mathfrak{H}$ and $\mathfrak{R}$}
\label{lib5table}
\end{table}

\subsection{Functorial Relationship to $\cat{Q}$}

Since $\cat{R}=\Set^{\cat{D}}$ and $\cat{Q}=\Set^{\cat{E}}$ are so similar as both functor categories and comma categories, natural questions to ask would be how different they are and how to connect them.  Firstly, the two categories are not equivalent as $\cat{Q}$ requires two objects to generate by Theorem \ref{generation-Q}, but $\cat{R}$ requires three by Theorem \ref{generation-R}.

On the other hand, any functor from $\cat{D}$ to $\cat{E}$ gives rise to a composition functor from $\Set^{\cat{E}}$ to $\Set^{\cat{D}}$ by \cite[Example A4.1.4]{elephant1}.  Moreover, \cite[Theorem I.3.7.2]{borceux} and its dual guarantee that such a composition functor admits both a left and a right adjoint functor via Kan extensions.

Specifically, consider the functor from $\cat{D}$ to $\cat{E}$ determined by $y\mapsto s$, $z\mapsto t$.
\[\begin{array}{ccc}
\cat{D}	&	\Rightarrow	&	\cat{E}\\
\xymatrix{
&	2\ar[dl]_{y}\ar[dr]^{z}\\
0	&	&	1\\
}
&	&
\xymatrix{
1\ar@/_/[d]_{s}\ar@/^/[d]^{t}\\
0\\
}\\
\end{array}\]
The corresponding composition functor $\Upsilon:\cat{Q}\to\cat{R}$ would have the following action.
\begin{itemize}
\item $\Upsilon\left(\overrightarrow{V}(Q),\overrightarrow{E}(Q),\sigma_{Q},\tau_{Q}\right)=\left(\overrightarrow{V}(Q),\overrightarrow{V}(Q),\overrightarrow{E}(Q),\sigma_{Q},\tau_{Q}\right)$,
\item $\Upsilon\left(\overrightarrow{V}(\phi),\overrightarrow{E}(\phi)\right)=\left(\overrightarrow{V}(\phi),\overrightarrow{V}(\phi),\overrightarrow{E}(\phi)\right)$.
\end{itemize}
Effectively, this functor converts directed edges into incidences, and duplicates vertices of the quiver as both vertices and edges of the incidence hypergraph.

While categorically natural, the combinatorial meaning of $\Upsilon$ is not immediately apparent, and will be discussed later.  The adjoints of $\Upsilon$, however, have discernible combinatorial meaning.  Via Kan extensions, $\Upsilon$ admits a right adjoint $\Upsilon^{\star}:\cat{R}\to\cat{Q}$ and a left adjoint $\Upsilon^{\diamond}:\cat{R}\to\cat{Q}$ with the following action on objects:
\begin{itemize}
\item $\Upsilon^{\star}(G)=\left(\check{V}(G)\times\check{E}(G),\check{V}(G)\times I(G)\times\check{E}(G),\sigma_{\Upsilon^{\star}(G)},\tau_{\Upsilon^{\star}(G)}\right)$, where $\sigma_{\Upsilon^{\star}(G)}(v,i,e)=\left(\varsigma_{G}(i),e\right)$ and $\tau_{\Upsilon^{\star}(G)}(v,i,e)=\left(v,\omega_{G}(i)\right)$;
\item $\Upsilon^{\diamond}(G)=\left(\check{V}(G)\sqcup\check{E}(G),I(G),\varpi_{1}\circ\varsigma_{G},\varpi_{2}\circ\omega_{G}\right)$, where $\varpi_{1}:\check{V}(G)\to\check{V}(G)\sqcup\check{E}(G)$ and $\varpi_{2}:\check{E}(G)\to\check{V}(G)\sqcup\check{E}(G)$ are the canonical inclusions into the disjoint union.
\end{itemize}
As can be seen through a simple example in Figure \ref{UpsilonDiamondStarExample}, $\Upsilon^{\diamond}$ encodes the bipartite incidence digraph, converting an incidence into a directed edge.  On the other hand, $\Upsilon^{\star}$ encodes the incidence matrix with extra structure.  Indeed, a simple calculation shows that a directed loop in $\Upsilon^{\star}(G)$ corresponds precisely to an incidence in $G$. 

\begin{figure}[H]
\centering
\includegraphics[scale=1]{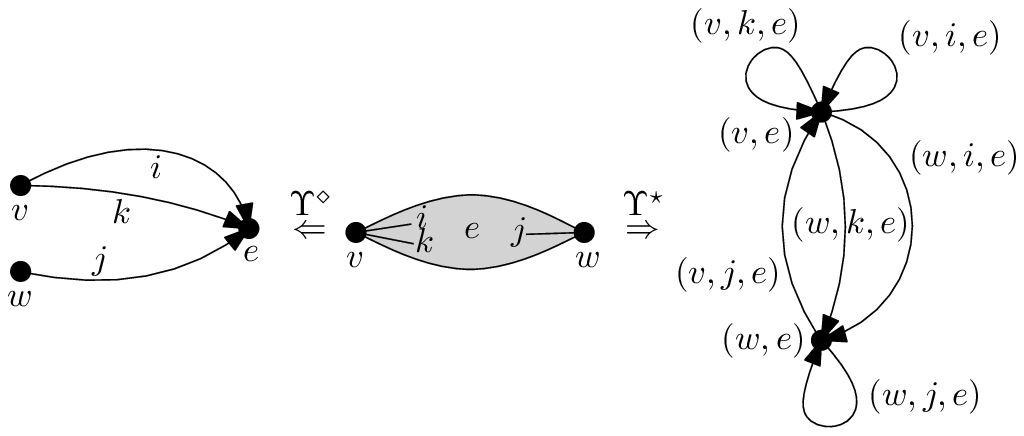}
\caption{The action of $\Upsilon^{\diamond}$ and $\Upsilon^{\star}$}
\label{UpsilonDiamondStarExample}
\end{figure}

These $\cat{R}$-$\cat{Q}$-inter-categorical functors are collected in Table \ref{lib6table}.

\begin{table}[H]
\centering
\renewcommand{\arraystretch}{1}
\begin{tabular}{|l|c|l|}
\hline
Functor & $\mathrm{Dom}\rightarrow \mathrm{Codom}$ & Note \\
\hline\hline
$\Upsilon $ & $\cat{Q}\rightarrow \cat{R}$ & Edge to incidence conversion functor. \\ 
$\Upsilon ^{\ast }$ & $\cat{R}\rightarrow \cat{Q}$ & Looped incidence matrix (with additional links). \\ 
$\Upsilon ^{\diamond}$ & $\cat{R}\rightarrow \cat{Q}$ & Bipartite incidence digraph. \\ \hline
\end{tabular}
\caption{A library of functors between $\mathfrak{Q}$ and $\mathfrak{R}$}
\label{lib6table}
\end{table}

Yet, what do the extra edges in $\Upsilon^{\star}(G)$ mean?  This structure is closely related to the exponential in $\cat{Q}$, as well as $\cat{R}$.  Recall that both $\cat{Q}$ and $\cat{R}$ are cartesian closed by \cite[Proposition A1.5.5]{elephant1}.  However, while several sources discuss and even compute examples of quiver exponentials (\cite{brown2008}, \cite{bumby1986}), none seem to have written down a concrete representation of the exponential for $\cat{Q}$ beyond the general construction (\cite{elephant1}, \cite{maclane2}).  A possible explanation for this is that, as seen below, the quiver exponential is intimately tied to the incidence hypergraph exponential via $\Upsilon$.

To explain, first consider the evaluation functors $\check{V}$, $\check{E}$, and $I$.  All of three are components of essential geometric morphisms, but the functors $\check{V}$ and $\check{E}$ are also logical functors, preserving the exponential operation.  However, $I$ fails to be logical.

\begin{lem}[Structure of $\check{V}$ \& $\check{E}$]\label{ve-geometric}
For $G\in\ob(\cat{R})$ and $S\in\ob(\Set)$, the Frobenius morphisms $\xymatrix{\check{V}^{\diamond}\left(\check{V}(G)\times S\right)\ar[r]^(0.6){\Phi_{V}} & G\prod\check{V}^{\diamond}(S)}\in\cat{R}$ and $\xymatrix{\check{E}^{\diamond}\left(\check{E}(G)\times S\right)\ar[r]^(0.6){\Phi_{E}} & G\prod\check{E}^{\diamond}(S)}\in\cat{R}$ are given by $\Phi_{V}=\left(id_{\check{V}(G)\times S},id_{\emptyset},id_{\emptyset}\right)$ and $\Phi_{E}=\left(id_{\emptyset},id_{\check{E}(G)\times S},id_{\emptyset}\right)$.  Moreover, both $\left(\check{V}^{\diamond},\check{V},\check{V}^{\star}\right)$ and $\left(\check{E}^{\diamond},\check{E},\check{E}^{\star}\right)$ are atomic geometric embeddings from $\Set$ to $\cat{R}$, which are not surjections.  Consequently, $\Set$ is equivalent to the slice categories $\cat{R}/\check{V}^{\diamond}(\{1\})$ and $\cat{R}/\check{E}^{\diamond}(\{1\})$.
\end{lem}

\begin{proof}

Observe that
\begin{itemize}
\item $\check{V}\check{V}^{\diamond}\left(\check{V}(G)\times S\right)=\check{V}(G)\times S=\check{V}\left(G\prod\check{V}^{\diamond}(S)\right)$,
\item $\check{E}\check{V}^{\diamond}\left(\check{V}(G)\times S\right)=\emptyset=\check{E}\left(G\prod\check{V}^{\diamond}(S)\right)$,
\item $I\check{V}^{\diamond}\left(\check{V}(G)\times S\right)=\emptyset=I\left(G\prod\check{V}^{\diamond}(S)\right)$.
\end{itemize}
Thus, $\check{E}\left(\Phi_{V}\right)=I\left(\Phi_{V}\right)=id_{\emptyset}$.  Applying $\check{V}$ to the defining diagram for $\Phi_{V}$ from \cite[Lemma A1.5.8]{elephant1} yields the following diagram in $\Set$.
\[\xymatrix{
\check{V}(G)\ar[d]_{id_{\check{V}(G)}}	&	&	\check{V}(G)\times S\ar[drr]^{\pi_{S}}\ar[ll]_{\pi_{\check{V}(G)}}\ar[d]_{\check{V}\left(\Phi_{V}\right)}\\
\check{V}(G)	&	&	\check{V}\left(G\prod\check{V}^{\diamond}(S)\right)\ar[rr]_{\pi_{S}}\ar[ll]^{\pi_{\check{V}(G)}}	&	&	S\\
}\]
Consequently, $\check{V}\left(\Phi_{V}\right)=id_{\check{V}(G)\times S}$.  Thus, $\Phi_{V}$ is an isomorphism, so $\check{V}$ is logical.

The unit $\xymatrix{S\ar[r]^(0.4){\eta_{S}} & \check{V}\check{V}^{\diamond}(S)}\in\Set$ is precisely $\eta_{S}=id_{S}$, showing that $\check{V}^{\diamond}$ is full and faithful by \cite[Proposition I.3.4.1]{borceux}.  Let $\xymatrix{[2]\ar[r]^{f} & [2]}\in\Set$ by $f(1):=2$ and $f(2):=1$.  Note that $\check{V}\check{E}^{\star}(f)=id_{\{1\}}=\check{V}\left(id_{\check{E}^{\star}\left([2]\right)}\right)$, despite $\check{E}^{\star}(f)\neq id_{\check{E}^{\star}\left([2]\right)}$.  Thus, $\check{V}$ is not faithful.  Therefore, $\left(\check{V}^{\diamond},V,\check{V}^{\star}\right)$ is an embedding but not a surjection.

The proof for $\check{E}$ is nearly identical.

\qed \end{proof}

\begin{lem}[Structure of $I$]\label{i-notlogical}
For $G\in\ob(\cat{R})$ and $S\in\ob(\Set)$, the Frobenius morphism $\xymatrix{I^{\diamond}\left(I(G)\times S\right)\ar[r]^(0.6){\Phi_{I}} & G\prod I^{\diamond}(S)}\in\cat{R}$ is given by $\Phi_{I}=\left(\varsigma_{G}\times id_{S},\omega_{G}\times id_{S},id_{I(G)\times S}\right)$.  Moreover, $\left(I^{\diamond},I,I^{\star}\right)$ is an essential geometric embedding from $\Set$ to $\cat{R}$, which is neither a surjection nor atomic.
\end{lem}

\begin{proof}

Observe that
\begin{itemize}
\item $\check{V}I^{\diamond}\left(I(G)\times S\right)=\check{E}I^{\diamond}\left(I(G)\times S\right)=II^{\diamond}\left(I(G)\times S\right)=I(G)\times S=I\left(G\prod I^{\diamond}(S)\right)$,
\item $\check{V}\left(G\prod I^{\diamond}(S)\right)=\check{V}(G)\times S$,
$\check{E}\left(G\prod I^{\diamond}(S)\right)=\check{E}(G)\times S$.
\end{itemize}
Applying $I$ to the defining diagram for $\Phi_{I}$ from \cite[Lemma A1.5.8]{elephant1} yields the following diagram in $\Set$.
\[\xymatrix{
I(G)\ar[d]_{id_{I(G)}}	&	&	I(G)\times S\ar[drr]^{\pi_{S}}\ar[ll]_{\pi_{I(G)}}\ar[d]_{I\left(\Phi_{I}\right)}\\
I(G)	&	&	I\left(G\prod I^{\diamond}(S)\right)\ar[rr]_{\pi_{S}}\ar[ll]^{\pi_{I(G)}}	&	&	S\\
}\]
Consequently, $I\left(\Phi_{I}\right)=id_{I(G)\times S}$.  On the other hand, applying $\check{V}$ gives the diagram below.
\[\xymatrix{
I(G)\ar[d]_{\varsigma_{G}}	&	&	I(G)\times S\ar[drr]^{\pi_{S}}\ar[ll]_{\pi_{I(G)}}\ar[d]_{\check{V}\left(\Phi_{I}\right)}\\
\check{V}(G)	&	&	\check{V}(G)\times S\ar[rr]_{\pi_{S}}\ar[ll]^{\pi_{\check{V}(G)}}	&	&	S\\
}\]
Thus, $\check{V}\left(\Phi_{I}\right)=\varsigma_{G}\times id_{S}$, and $\check{E}\left(\Phi_{I}\right)=\omega_{G}\times id_{S}$ by a similar argument.  Note that $\Phi_{I}$ is not monic when $G=I^{\star}\left([2]\right)$ and $S=[2]$.  Therefore, $I$ is not logical.

The geometric morphism $\left(I^{\diamond},I,I^{\star}\right)$ is an embedding but not a surjection by the same proof as Lemma \ref{ve-geometric}.

\qed \end{proof}

Next, consider the incidence hypergraph exponential, which is easily constructed.  As $\cat{R}$ is known to be a topos, the exponential $H^{G}$ is known to exist for all $G,H\in\ob(\cat{R})$.  As $\check{V}$ and $\check{E}$ are logical, one has that $\check{V}\left(H^{G}\right)\cong_{\Set}{\check{V}(H)}^{\check{V}(G)}=\Set\left(\check{V}(G),\check{V}(H)\right)$ and $\check{E}\left(H^{G}\right)\cong_{\Set}{\check{E}(H)}^{\check{E}(G)}=\Set\left(\check{E}(G),\check{E}(H)\right)$.  Using Lemma \ref{i-notlogical}, and the fact that $I^{\diamond}(\{1\})\cong_{\cat{R}}I^{\star}(\{1\})$ is terminal, one has the following bijection:
\[\begin{array}{rcl}
I\left(H^{G}\right)
&	\cong_{\Set}	&	\Set\left(\{1\},I\left(H^{G}\right)\right)
\cong_{\Set}\cat{R}\left(I^{\diamond}\left(\{1\}\right),H^{G}\right)
\cong_{\Set}\cat{R}\left(G{\prod}^{\cat{R}}I^{\diamond}\left(\{1\}\right),H\right)\\
&	\cong_{\Set}	&	\cat{R}\left(G,H\right).\\
\end{array}\]
Thus, the exponential of an incidence hypergraph can be constructed as described below.

\begin{defn}[Incidence hypergraph exponential]
Given incidence hypergraphs $G,H$, define the incidence hypergraph $H^{G}$ in the following way:
\begin{itemize}
\item $\check{V}\left(H^{G}\right):=\Set\left(\check{V}(G),\check{V}(H)\right)$, $\check{E}\left(H^{G}\right):=\Set\left(\check{E}(G),\check{E}(H)\right)$, $I\left(H^{G}\right):=\cat{R}(G,H)$,
\item $\varsigma_{H^{G}}(\phi):=\check{V}(\phi)$, $\omega_{H^{G}}(\phi):=\check{E}(\phi)$.
\end{itemize}
Define $\xymatrix{G{\prod}^{\cat{R}}H^{G}\ar[r]^(0.7){\epsilon_{H}^{G}} & H}\in\cat{R}$ by $\check{V}\left(\epsilon_{H}^{G}\right)(v,f):=f(v)$, $\check{E}\left(\epsilon_{H}^{G}\right)(e,g):=g(e)$, $I\left(\epsilon_{H}^{G}\right)(i,\phi):=I(\phi)(i)$.
\end{defn}

\begin{thm}[Universal property of the exponential, $\cat{R}$]\label{exponential-r}
Given $\xymatrix{G{\prod}^{\cat{R}}K\ar[r]^(0.65){\psi} & H}\in\cat{R}$, there is a unique $\xymatrix{K\ar[r]^{\hat{\psi}} & H^{G}}\in\cat{R}$ such that $\epsilon_{H}^{G}\circ\left(G{\prod}^{\cat{R}}\hat{\psi}\right)=\psi$.
\end{thm}

\begin{proof}
Let $\xymatrix{G{\prod}^{\cat{R}}I^{\diamond}(\{1\})\ar[r]^(0.7){\rho_{G}} & G}\in\cat{R}$ be the right unitor isomorphism for the cartesian monoidal structure on $\cat{R}$, which is given by $(v,1)\mapsto v$, $(e,1)\mapsto e$, $(i,1)\mapsto i$.  Given $j\in I(K)$, define $\alpha_{j}:\{1\}\to I(K)$ by $\alpha_{j}(1):=j$.  There is a unique $\xymatrix{I^{\diamond}(\{1\})\ar[r]^(0.6){\hat{\alpha}_{j}} & K}\in\cat{R}$ such that $I\left(\hat{\alpha}_{j}\right)=\alpha_{j}$.  Define $\xymatrix{K\ar[r]^{\hat{\psi}} & H^{G}}\in\cat{R}$ by
\begin{itemize}
\item $\check{V}\left(\hat{\psi}\right)(w)(v):=\check{V}(\psi)(v,w)$, $\check{E}\left(\hat{\psi}\right)(f)(e):=\check{E}(\psi)(e,f)$,
\item $I\left(\hat{\psi}\right)(j):=\psi\circ\left(G{\prod}^{\cat{R}}\hat{\alpha}_{j}\right)\circ\rho_{G}^{-1}$.
\end{itemize}
\qed \end{proof}

Turning attention to $\cat{Q}$, one can repeat Lemmas \ref{ve-geometric} and \ref{i-notlogical} to yield the analogous results.

\begin{lem}[Structure of $\overrightarrow{V}$]
For $Q\in\ob(\cat{Q})$ and $S\in\ob(\Set)$, the Frobenius morphism $\xymatrix{\overrightarrow{V}^{\diamond}\left(\overrightarrow{V}(Q)\times S\right)\ar[r]^(0.6){\varphi_{V}} & Q\prod\overrightarrow{V}^{\diamond}(S)}\in\cat{Q}$ is given by $\varphi_{V}=\left(id_{\overrightarrow{V}(Q)\times S},id_{\emptyset}\right)$.  Moreover, $\left(\overrightarrow{V}^{\diamond},\overrightarrow{V},\overrightarrow{V}^{\star}\right)$ is an atomic geometric embedding from $\Set$ to $\cat{Q}$, which is not a surjection.  Consequently, $\Set$ is equivalent to the slice category $\cat{Q}/\overrightarrow{V}^{\diamond}(\{1\})$.
\end{lem}

\begin{lem}[Structure of $\overrightarrow{E}$]
For $Q\in\ob(\cat{Q})$ and $S\in\ob(\Set)$, the Frobenius morphism $\xymatrix{\overrightarrow{E}^{\diamond}\left(\overrightarrow{E}(Q)\times S\right)\ar[r]^(0.6){\varphi_{E}} & Q\prod\overrightarrow{E}^{\diamond}(S)}\in\cat{Q}$ is given by $\varphi_{E}=\left(\left(\sigma_{Q}\times\varpi_{1}\right)\sqcup\left(\tau_{Q}\times\varpi_{2}\right),id_{\overrightarrow{E}(Q)\times S}\right)$.  Moreover, $\left(\overrightarrow{E}^{\diamond},\overrightarrow{E},\overrightarrow{E}^{\star}\right)$ is an essential geometric embedding from $\Set$ to $\cat{Q}$, which is neither a surjection nor atomic.
\end{lem}

On the other hand, the functor $\Upsilon$ itself is a component of a geometric morphism, and is also logical, which blends $\cat{Q}$ and $\cat{R}$ deeply together.

\begin{thm}[Nature of $\Upsilon$]\label{natureofupsilon}
For $Q\in\ob(\cat{Q})$ and $G\in\ob(\cat{R})$, the Frobenius morphism $\xymatrix{\Upsilon^{\diamond}\left(\Upsilon(Q)\prod G\right)\ar[r]^(0.6){\Phi} & Q\prod\Upsilon^{\diamond}(G)}\in\cat{Q}$ is given by
\begin{itemize}
\item $\overrightarrow{V}(\Phi)(n,v,x)=(v,n,x)$,
\item $\overrightarrow{E}(\Phi)=id_{\overrightarrow{E}(Q)\times I(G)}$.
\end{itemize}
Moreover, $\left(\Upsilon^{\diamond},\Upsilon,\Upsilon^{\star}\right)$ is an atomic geometric surjection from $\cat{R}$ to $\cat{Q}$, which is not an embedding.  Furthermore, $\Upsilon^{\diamond}$ is faithful, so $\cat{R}$ is equivalent to the slice category $\cat{Q}/\overrightarrow{E}^{\diamond}(\{1\})$.
\end{thm}

\begin{proof}

Observe that
\begin{itemize}
\item $\overrightarrow{E}\Upsilon^{\diamond}\left(\Upsilon(Q)\prod G\right)=\overrightarrow{E}(Q)\times I(G)=\overrightarrow{E}\left(Q\prod\Upsilon^{\diamond}(G)\right)$,
\item $\overrightarrow{V}\Upsilon^{\diamond}\left(\Upsilon(Q)\prod G\right)=\left(\{1\}\times\overrightarrow{V}(Q)\times\check{V}(G)\right)\cup\left(\{2\}\times\overrightarrow{V}(Q)\times\check{E}(G)\right)$,
\item $\overrightarrow{V}\left(Q\prod\Upsilon^{\diamond}(G)\right)=\overrightarrow{V}(Q)\times\left(\left(\{1\}\times\check{V}(G)\right)\cup\left(\{2\}\times\check{E}(G)\right)\right)$.
\end{itemize}
Applying $\overrightarrow{E}$ to the defining diagram for $\Phi$ from \cite[Lemma A1.5.8]{elephant1} yields the following diagram in $\Set$.
\[\xymatrix{
\overrightarrow{E}(Q)\ar[d]_{id_{\overrightarrow{E}(Q)}}	&	&	\overrightarrow{E}(Q)\times I(G)\ar[drr]^{\pi_{I(G)}}\ar[ll]_{\pi_{\overrightarrow{E}(Q)}}\ar[d]_{\overrightarrow{E}\left(\Phi\right)}\\
\overrightarrow{E}(Q)	&	&	\overrightarrow{E}\left(Q\prod\Upsilon^{\diamond}(G)\right)\ar[rr]_{\pi_{I(G)}}\ar[ll]^{\pi_{\overrightarrow{E}(Q)}}	&	&	I(G)\\
}\]
Consequently, $\overrightarrow{E}\left(\Phi\right)=id_{\overrightarrow{E}(Q)\times I(G)}$.  Applying $\overrightarrow{V}$ yields the following diagram, where $\overrightarrow{V}\left(\epsilon_{Q}\right)(n,v)=v$.
\[\xymatrix{
\{1,2\}\times\overrightarrow{V}(Q)\ar[d]_{\overrightarrow{V}\left(\epsilon_{Q}\right)}	&	&	\overrightarrow{V}\Upsilon^{\diamond}\left(\Upsilon(Q)\prod G\right)\ar[drr]^{\overrightarrow{V}\Upsilon^{\diamond}\left(\pi_{G}\right)}\ar[ll]_{\overrightarrow{V}\Upsilon^{\diamond}\left(\pi_{\Upsilon(G)}\right)}\ar[d]_{\overrightarrow{V}\left(\Phi\right)}\\
\overrightarrow{V}(Q)	&	&	\overrightarrow{V}\left(Q\prod\Upsilon^{\diamond}(G)\right)\ar[rr]_(0.4){\pi_{\check{V}\Upsilon^{\diamond}(G)}}\ar[ll]^{\pi_{\overrightarrow{V}(Q)}}	&	&	\left(\{1\}\times\check{V}(G)\right)\cup\left(\{2\}\times\check{E}(G)\right)\\
}\]
A diagram chase shows that $\overrightarrow{V}(\Phi)(n,v,x)=(v,n,x)$.  Thus, $\Phi$ is an isomorphism, so $\Upsilon$ is logical.

A quick check shows that $I\Upsilon=\overrightarrow{E}$ and $\check{V}\Upsilon=\check{E}\Upsilon=\overrightarrow{V}$.  If $\xymatrix{Q\ar@/^/[r]^{\alpha}\ar@/_/[r]_{\beta} & R}\in\cat{Q}$ satisfy that $\Upsilon(\alpha)=\Upsilon(\beta)$, then $\overrightarrow{V}(\alpha)=\check{V}\Upsilon(\alpha)=\check{V}\Upsilon(\beta)=\overrightarrow{V}(\beta)$ and $\overrightarrow{E}(\alpha)=I\Upsilon(\alpha)=I\Upsilon(\beta)=\overrightarrow{E}(\beta)$.  Thus, $\alpha=\beta$, showing $\Upsilon$ is faithful.  Therefore, $\left(\Upsilon^{\diamond},\Upsilon,\Upsilon^{\star}\right)$ is a surjection.

If $\xymatrix{G\ar@/^/[r]^{\alpha}\ar@/_/[r]_{\beta} & H}\in\cat{R}$ satisfy that $\Upsilon^{\diamond}(\alpha)=\Upsilon^{\diamond}(\beta)$, then
\begin{itemize}
\item $\check{V}(\alpha)=\overrightarrow{V}\Upsilon^{\diamond}(\alpha)\circ\varpi_{\check{V}(G)}=\overrightarrow{V}\Upsilon^{\diamond}(\beta)\circ\varpi_{\check{V}(G)}=\check{V}(\beta)$,
\item $\check{E}(\alpha)=\overrightarrow{V}\Upsilon^{\diamond}(\alpha)\circ\varpi_{\check{E}(G)}=\overrightarrow{V}\Upsilon^{\diamond}(\beta)\circ\varpi_{\check{E}(G)}=\check{E}(\beta)$,
\item $I(\alpha)=\overrightarrow{E}\Upsilon^{\diamond}(\alpha)=\overrightarrow{E}\Upsilon^{\diamond}(\beta)=I(\beta)$.
\end{itemize}
Thus, $\alpha=\beta$, showing that $\Upsilon^{\diamond}$ is faithful.  Therefore, $\cat{R}$ is equivalent to $\cat{Q}/\Upsilon^{\diamond}I^{\diamond}\left(\{1\}\right)\cong\cat{Q}/\overrightarrow{E}^{\diamond}\left(\{1\}\right)$ by \cite[Proposition A2.3.8]{elephant1}.

The counit $\xymatrix{\Upsilon\Upsilon^{\star}(G)\ar[r]^(0.6){\Theta_{G}} & G}\in\cat{R}$ is given by $\check{V}\left(\Theta_{G}\right)(v,e)=v$, $\check{E}\left(\Theta_{G}\right)(v,e)=e$, and $I\left(\Theta_{G}\right)(v,i,e)=i$.  Note that $\Theta_{G}$ is not monic when $G=I^{\diamond}\left([2]\right)$.  Therefore, $\left(\Upsilon^{\diamond},\Upsilon,\Upsilon^{\star}\right)$ is not an embedding.

\qed \end{proof}

An immediate corollary shows that the action of $\Upsilon^{\diamond}\Upsilon$ is identical to the product with $\overrightarrow{P}_{1}\cong_{\cat{Q}}\overrightarrow{E}^{\diamond}(\{1\})$, and the action of $\Upsilon^{\star}\Upsilon$ can be shown to be identical raising to $\overrightarrow{P}_{1}$.  Consequently, $\Upsilon^{\star}$ can be considered an extension of $\Box^{\overrightarrow{P}_{1}}$ to $\cat{R}$.

\begin{cor}[Actions of $\Upsilon^{\diamond}\Upsilon$ and $\Upsilon^{\star}\Upsilon$]\label{updiaup}
For $Q\in\ob(\cat{Q})$, one has the following natural isomorphisms:
\begin{enumerate}
\item $\Upsilon^{\diamond}\Upsilon(Q) \cong_{\cat{Q}} Q{\prod}^{\cat{Q}}\overrightarrow{P}_{1}$,
\item $\Upsilon^{\star}\Upsilon(Q) \cong_{\cat{Q}} Q^{\overrightarrow{P}_{1}}$.
\end{enumerate}
\end{cor}

\begin{proof}
Thus, the following isomorphisms are natural in $Q$ and $K$.\\
\textit{Proof of (1):}
\[\begin{array}{rcl}
\Upsilon^{\diamond}\Upsilon(Q)
&	\cong_{\cat{Q}}	&	\Upsilon^{\diamond}\left(\Upsilon(Q){\prod}^{\cat{R}}I^{\diamond}(\{1\})\right)
\cong_{\cat{Q}}Q{\prod}^{\cat{Q}}\Upsilon^{\diamond}I^{\diamond}(\{1\})
\cong_{\cat{Q}}Q{\prod}^{\cat{Q}}\left(I\Upsilon\right)^{\diamond}(\{1\})\\
&	\cong_{\cat{Q}}	&	Q{\prod}^{\cat{Q}}\overrightarrow{E}^{\diamond}(\{1\})
\cong_{\cat{Q}}Q{\prod}^{\cat{Q}}\overrightarrow{P}_{1}\\
\end{array}\]
\textit{Proof of (2):}
\[\begin{array}{rcl}
\cat{Q}\left(K,\Upsilon^{\star}\Upsilon(Q)\right)
&	\cong_{\Set}	&	\cat{Q}\left(\Upsilon(K),\Upsilon(Q)\right)
\cong_{\Set}\cat{Q}\left(\Upsilon^{\diamond}\Upsilon(K),Q\right)
\cong_{\Set}\cat{Q}\left(\overrightarrow{P}_{1}{\prod}^{\cat{Q}}K,Q\right)
\cong_{\Set}\cat{Q}\left(K,Q^{\overrightarrow{P}_{1}}\right)
\end{array}\]
\qed \end{proof}

Now, the quiver exponential can be constructed.  As $\cat{Q}$ is known to be a topos, the exponential $R^{Q}$ is known to exist for all $Q,R\in\ob(\cat{Q})$.  As $\overrightarrow{V}$ is logical, one has that $\overrightarrow{V}\left(R^{Q}\right)\cong_{\Set}{\overrightarrow{V}(R)}^{\overrightarrow{V}(Q)}=\Set\left(\overrightarrow{V}(Q),\overrightarrow{V}(R)\right)$.  Using Corollary \ref{updiaup}, one has the following bijection:
\[\begin{array}{rcl}
\overrightarrow{E}\left(R^{Q}\right)
&	\cong_{\Set}	&	\Set\left(\{1\},\overrightarrow{E}\left(R^{Q}\right)\right)
\cong_{\Set}\cat{Q}\left(\overrightarrow{E}^{\diamond}\left(\{1\}\right),R^{Q}\right)
\cong_{\Set}\cat{Q}\left(Q{\prod}^{\cat{Q}}\overrightarrow{E}^{\diamond}\left(\{1\}\right),R\right)\\
&	\cong_{\Set}	&	\cat{Q}\left(\Upsilon^{\diamond}\Upsilon(Q),R\right)
\cong_{\Set}\cat{R}\left(\Upsilon(Q),\Upsilon(R)\right)\\
\end{array}\]

This last calculation reveals a very peculiar truth.  The edges of the quiver exponential coincide with incidence hypergraph homomorphims involving $\Upsilon$, which might explain the difficulty in describing the quiver exponential solely by using only $\cat{Q}$.  Indeed, one can now concretely construct the quiver exponential in direct analogy to the exponential in $\cat{R}$.

\begin{defn}[Quiver exponential]\label{quiverexponential}
Given quivers $Q,R$, define the quiver $R^{Q}$ in the following way:
\begin{itemize}
\item $\overrightarrow{V}\left(R^{Q}\right):=\Set\left(\overrightarrow{V}(Q),\overrightarrow{V}(R)\right)$, $\overrightarrow{E}\left(R^{Q}\right):=\cat{R}\left(\Upsilon(Q),\Upsilon(R)\right)$,
\item $\sigma_{R^{Q}}(\phi):=\check{V}(\phi)$, $\tau_{R^{Q}}(\phi):=\check{E}(\phi)$.
\end{itemize}
Define $\xymatrix{Q{\prod}^{\cat{Q}}R^{G}\ar[r]^(0.7){\varepsilon_{R}^{Q}} & R}\in\cat{R}$ by $\overrightarrow{V}\left(\varepsilon_{R}^{Q}\right)(v,f):=f(v)$, $\overrightarrow{E}\left(\varepsilon_{R}^{Q}\right)(e,\phi):=I(\phi)(e)$.
\end{defn}

The proof of the following characterization is nearly identical to Theorem \ref{exponential-r} and will be omitted.

\begin{thm}[Universal property of the exponential, $\cat{Q}$]\label{qexponential}
Given $\xymatrix{Q{\prod}^{\cat{Q}}K\ar[r]^(0.65){\psi} & R}\in\cat{Q}$, there is a unique $\xymatrix{K\ar[r]^{\hat{\psi}} & R^{Q}}\in\cat{Q}$ such that $\varepsilon_{R}^{Q}\circ\left(Q{\prod}^{\cat{Q}}\hat{\psi}\right)=\psi$.
\end{thm}

With the exponential objects concretely constructed and interwoven, $\Upsilon$ provides combinatorial context to quiver exponentials by passing to $\cat{R}$, as seen in Theorem \ref{natureofupsilon}. Contrast this to the classical definition of the exponential of simple digraphs.

\begin{defn}[{Classical digraph exponential, \cite[p.\ 46]{GandH}}]
Given simple digraphs $Q$ and $R$, the digraph exponential $[Q,R]$ is defined by
\begin{itemize}
\item $\overrightarrow{V}[Q,R]:=\Set\left(\overrightarrow{V}(Q),\overrightarrow{V}(R)\right)$,
\item $\overrightarrow{E}[Q,R]:=\left\{(f,g)\in\overrightarrow{V}[Q,R]\times\overrightarrow{V}[Q,R]:\left(f(v),g(v)\right)\in\overrightarrow{E}(R)\forall v\in\overrightarrow{V}(Q)\right\}$,
\item $\sigma_{[Q,R]}(f,g):=f$, $\tau_{[Q,R]}(f,g):=g$.
\end{itemize}
The operation $[\cdot,\cdot\cdot]$ is well-known to be the exponential in the category of simple digraphs \cite[Corollary 2.18 \& Proposition 2.19]{GandH}.
\end{defn}

Be aware that the above construction only yields a simple digraph, as opposed to the Definition \ref{quiverexponential}.

Using the classical definition of the quiver exponential $\left[\overrightarrow{P}_{1},\overrightarrow{P}_{1}\right]$ with base vertices $\{u,v\}$ and exponent vertices $\{x,y\}$. The vertices correspond to the morphisms $f_1:x \rightarrow u, y\rightarrow u$; $f_2:x \rightarrow u, y \rightarrow v$; $f_3:x \rightarrow v,y \rightarrow u$; and $f_4:x \rightarrow v,y \rightarrow v$. While the edges are $(f_1,f_2):(x,y) \rightarrow (u,v)$; $(f_1,f_4):(x,y) \rightarrow (u,v)$; $(f_2,f_4):(x,y) \rightarrow (u,v)$; $(f_2,f_2):(x,y) \rightarrow (u,v)$. This is depicted in Figure \ref{ClassicalQuiverExp}.

\begin{figure}[H]
\centering
\includegraphics[scale=1]{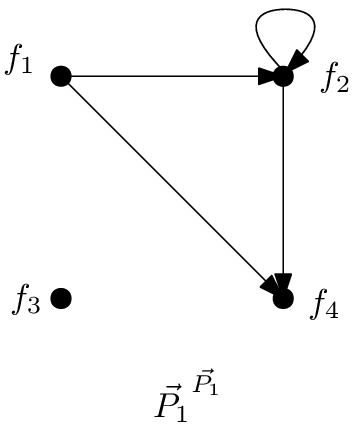}
\caption{The classical quiver exponential of a directed path to itself}
\label{ClassicalQuiverExp}
\end{figure}

However, the edges of the quiver exponential manifest as homomorphisms in $\cat{R}$. As an illustration, consider the action of $\Upsilon$ on the two paths from the classical quiver exponential example and consider all homomorphisms from $\Upsilon\left({\overrightarrow{P}_{1}}\right) \rightarrow \Upsilon\left({\overrightarrow{P}_{1}}\right)$, as in Figure \ref{UpsilonQuiverExp}.

\begin{figure}[H]
\centering
\includegraphics[scale=1]{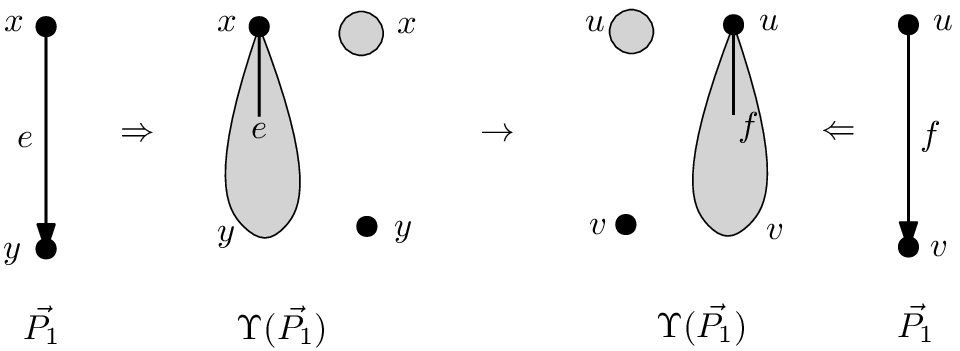}
\caption{The classical quiver exponential of a directed path to itself}
\label{UpsilonQuiverExp}
\end{figure}

The edges of ${\overrightarrow{P}_{1}}^{\overrightarrow{P}_{1}}$ in $\cat{Q}$ are homomorphisms in $\cat{R}$ from $\Upsilon\left(\overrightarrow{P}_{1}\right)$ to itself by Theorem \ref{qexponential}, noting that $\Upsilon\left(\overrightarrow{P_1}\right)$ is the disjoint union of the generators of $\cat{R}$.  This is made apparent in Figure \ref{QExpontAsUpsilon} through the pullback into the slice category $\cat{Q}/\overrightarrow{P}_{1}$, which is the bipartite incidence graph from $\cat{R}$. Observe that the four incidences of the $2$-cycle in $\Upsilon\left({\overrightarrow{P}_{1}}^{\overrightarrow{P}_{1}}\right)\cong_{\cat{R}}\Upsilon\left(\overrightarrow{P}_{1}\right)^{\Upsilon\left(\overrightarrow{P}_{1}\right)}$ manifests as an orientation of a $4$-cycle in the bipartite incidence graph.

\begin{figure}[H]
\centering
\includegraphics[scale=1]{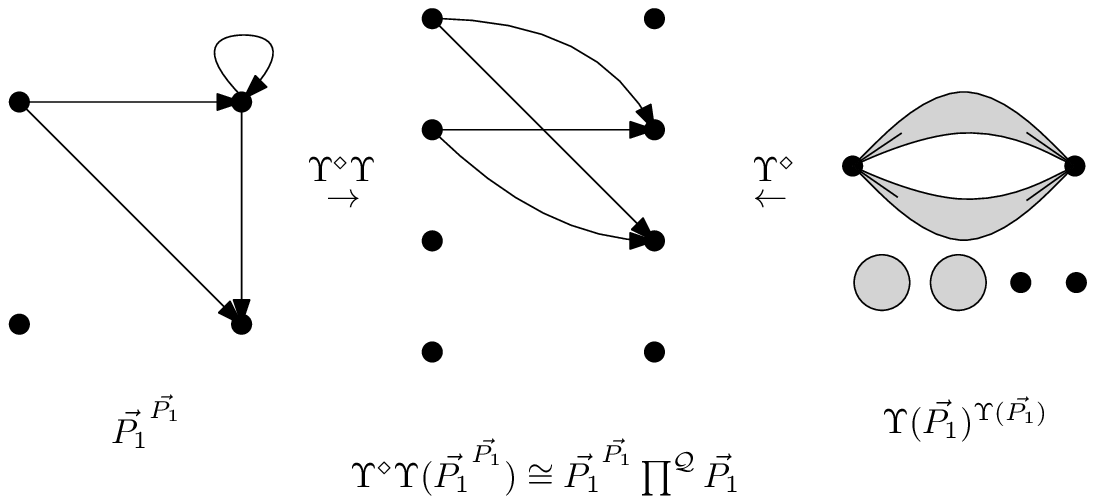}
\caption{Exponentials in $\cat{Q}$ and $\cat{R}$}
\label{QExpontAsUpsilon}
\end{figure}

\section{Conclusion}\label{library}

The table below summarizes the major properties of the categories studied in this paper.  Here, $I_{k}:=I^{\star}([k])$ is the incidence hypergraph having $k$ parallel incidences between a single vertex and a single edge, $\overrightarrow{B}_{k}:=\overrightarrow{E}^{\star}([k])$ is the bouquet of $k$ directed loops at a single vertex, and $E_{k}$ is the set-system hypergraph with a single $k$-edge.  Also, ``$v$'' (resp.\ ``$e$'') represents an isolated vertex (resp.\ loose edge), constructed appropriately for each category.  For comparison, $\Set$ itself has been included, which is generated by a singleton set \cite[Example I.4.5.17.a]{borceux} and has every object projective \cite[Example I.4.6.8.a]{borceux}.

\begin{table}[H]
\centering
\renewcommand{\arraystretch}{1.1}
\begin{tabular}{|c||c|c|c|c|c|}
\hline
Properties &  $\Set$ & $\cat{R}$ & $\cat{Q}$ & $\cat{M}$ & $\cat{H}$ \\ \hline\hline
Limits & Yes & Yes & Yes & Yes & Yes \\ \hline
Colimits & Yes & Yes & Yes & Yes & Yes \\ \hline
Subobject classifier & Yes & Yes & Yes & Yes & Yes \\ \hline
Cartesian closed & Yes & Yes & Yes & No & No \\ \hline
Projective cover & Yes & Yes & Yes & Yes & No \\ \hline
Minimal family of generators & $\{1\}$ & $v,e,I_1$ & $v,\overrightarrow{P}_1$ & $v,P_1$ & class \\ \hline
Terminal object & $\{1\}$ & $I_1$ & $\overrightarrow{B}_1$ & $E_1$ & $E_1{\coprod}^{\cat{H}}E_0$ \\ \hline
\end{tabular}
\caption{Comparison of $\Set$, $\cat{R}$, $\cat{Q}$, $\cat{M}$, and $\cat{H}$}
\end{table}

Observe that as one moves from the left of the table to the right, the categorical issues compound in the following ways:
\begin{itemize}
\item $\Set$ is generated by a single object, but $\cat{R}$ is generated by three;
\item the terminal object of $\cat{R}$ is also a generator, but not so in $\cat{Q}$;
\item $\cat{Q}$ is cartesian closed, but $\cat{M}$ is not;
\item $\cat{M}$ has projective covers for every object, a contiguous terminal object, and a pair of generators, but $\cat{H}$ lacks projective covers for non-projective objects, has a disconnected terminal object, and has a proper class of generators.
\end{itemize}
Notably, $\cat{R}$ reflects the properties of $\Set$ more closely than $\cat{Q}$ does. In particular, the exponential objects in $\cat{R}$ are determined by morphisms in $\cat{R}$, while the exponential objects in $\cat{Q}$ are also determined by morphisms in $\cat{R}$.

That said, Theorem \ref{natureofupsilon} gives an interesting interplay between the categories $\cat{Q}$ and $\cat{R}$.  One can quickly show that $\Upsilon$ is not full.  Specifically, $\cat{R}$ allows for the vertex and edge maps to differ, which is not allowed under the image of $\Upsilon$.  Consequently, $\Upsilon$ embeds the topos structure of $\cat{Q}$ into $\cat{R}$, but $\cat{R}$ allows for more morphisms than would be allowed in $\cat{Q}$.

On the other hand, $\cat{R}$ is equivalent to $\cat{Q}/\overrightarrow{P}_{1}$, a category of directed bipartite graphs, which demonstrates that the structure of an incidence hypergraph arises entirely from its bipartite incidence digraph.  However, $\cat{Q}$ can allow more homomorphisms than those in the slice category, noting that $\Upsilon^{\diamond}$ is also not full.  In particular, isolated vertices and loose edges in $\cat{R}$ both become isolated vertices in $\cat{Q}$ under $\Upsilon$, which may be mapped freely, rather than remaining in a fixed part of the bipartition of the vertices.

Categorical relationships such as this may be exploited to answer some combinatorial questions.  For example, the relationship above states that questions about bipartite graphs could be reformulated as questions about incidence hypergraphs, and perhaps solved using techniques of the latter \cite{Chen2017}, \cite{OHHar}, \cite{Reff1}, \cite{AH1}, \cite{OHMTT}, \cite{OH1}, \cite{Shi1}.

A question that can be immediately answered is from \cite{brown2008}.  In said paper's introduction, the authors make the analogy between sets and graphs, posing the question, ``What then should be a `member' of a graph?''  However, the answer for both $\cat{Q}$ and $\cat{R}$ comes from the theory of topoi, where a membership relation can be constructed in general via the exponential and subobject classifier \cite[Definition III.5.4.8]{borceux}.  Indeed, \cite[Example III.5.4.10.a]{borceux} demonstrates that the construction yields precisely the notion of membership to a set.

Understanding the combinatorial meaning of this membership relation, or other categorical concepts of the like, in $\cat{Q}$ and $\cat{R}$ might inform the analogy between sets and graphs, and address open graph-theoretic questions.  Consideration of such combinatorial meaning will be left to subsequent works.

\bibliographystyle{spmpsci}

\providecommand{\bysame}{\leavevmode\hbox to3em{\hrulefill}\thinspace}
\providecommand{\MR}{\relax\ifhmode\unskip\space\fi MR }

\providecommand{\MRhref}[2]{%
  \href{http://www.ams.org/mathscinet-getitem?mr=#1}{#2}
}
\providecommand{\href}[2]{#2}

\end{document}